\documentclass{amsart}

\newtheorem{theorem}{Theorem}[section]
\newtheorem{lemma}[theorem]{Lemma}
\newtheorem{corollary}[theorem]{Corollary}
\newtheorem{proposition}[theorem]{Proposition}

\theoremstyle{definition}

\newtheorem{example}[theorem]{Example}

\theoremstyle{remark}
\newtheorem{remark}[theorem]{Remark}

\everymath{\displaystyle}

\numberwithin{equation}{section}

\begin{document}

\title[Gradient estimates and volume doubling]{Gradient estimates and volume doubling for locally finite weighted graphs with $CD\psi(n,-K)$ condition}

\
\thanks{}

\author{Qianwei Zhang}
\address{School of Mathematical Sciences, Fudan University, Shanghai 200433, China}
\curraddr{}
\email{qwzhang@m.fudan.edu.cn}
\thanks{}

\subjclass[2020]{Primary }


\date{}

\dedicatory{}

\begin{abstract}
We study gradient estimates and volume growth on locally finite weighted graphs satisfying the $CD\psi(n,-K)$ condition with $K\geq0$. We establish a variational inequality for the heat semigroup and derive from it a family of Li-Yau type gradient estimates. Furthermore, under suitable assumptions on $\psi$, we establish a curvature dependent heat retention estimate for metric balls. Combined with a heat kernel Harnack inequality obtained from the established gradient estimate, this yields a curvature dependent exponential volume doubling estimate 
\begin{equation*}
    V(x,2r)\leq C e^{c\sqrt{K}r}V(x,r).
\end{equation*}
When $K=0$, the result reduces to a uniform volume doubling and further implies that the bottom of the spectrum of $-\Delta$ vanishes on infinite graphs. 
\end{abstract}

\maketitle
\section{Introduction}
Gradient estimates are a fundamental tool in geometric analysis. In 1975, Cheng-Yau \cite{Cheng-Yau} established gradient estimates for harmonic functions on complete Riemannian manifolds under lower Ricci curvature bounds. In 1986, Li-Yau
\cite{Li-Yau} obtained a global differential gradient estimate for positive solutions of the heat equation $\left(\Delta-\partial_{t}\right)u=0$. More precisely, if the Ricci curvature of an $n$-dimensional complete
Riemannian manifold is bounded below by $-K$ with $K\geq0$, then for every $\alpha>1$,
\begin{equation}
    \left |\nabla \log u \right |^{2}-\alpha\frac{\partial_{t}u}{u}\leq\frac{n\alpha^{2}K}{2(\alpha-1)}+\frac{n\alpha^{2}}{2t}, 
\label{chu}
\end{equation}
When $K=0$, letting $\alpha\to1+$ gives
\begin{equation}
    \left |\nabla \log u \right |^{2}-\frac{\partial_{t}u}{u}\leq\frac{n}{2t},
\label{li-yau classical}
\end{equation}
Under negative Ricci curvature bounds, several refinements of the Li-Yau estimate have subsequently been obtained. In particular, Bakry-Qian \cite{Bakry-Qian} in 1999 derived the time dependent estimate
\begin{equation}
    \left|\nabla\log u\right|^{2}-\left(1+\frac{2Kt}{3}\right)\frac{\partial_tu}{u}\leq\frac{n}{2t}+\frac{nK}{2}\left(1+\frac{Kt}{3}\right).
\label{Bakry-Qian manifold}
\end{equation}
A sharper time dependent estimate was obtained by Li-Xu \cite{Li-Xu}
\begin{equation}
    |\nabla \log u|^2
    -
    \left(1+\coth(Kt)-Kt\,\operatorname{csch}^2(Kt)\right)\frac{\partial_tu}{u}
    \leq
    \frac{nK}{2}\bigl(\coth(Kt)+1\bigr). 
\label{Li-Xu manifold}
\end{equation}
Further refinements were obtained, for example, by Zhang \cite{Zhang}.

The extension of gradient estimates to graphs involves additional difficulties, most notably the absence of a chain rule for the discrete Laplacian. As a result, the logarithmic transformation used in the manifold setting cannot in general be applied directly, and discrete gradient estimates are often formulated in terms of special functions such as the square root or, more generally, a concave function. Another issue is that several nonequivalent notions of curvature are available on graphs. In this paper, we work within the Bakry-\'Emery curvature dimension framework. Bakry-\'Emery \cite{Bakry-Emery} introduced the $\Gamma$-calculus and the curvature dimension criterion for diffusion semigroups in 1985. On a Riemannian manifold, the corresponding
curvature dimension inequality takes the form
\begin{equation*}
   \Gamma_2\left(f\right)\geq\frac{1}{n}\left(\Delta f\right)^2+K\Gamma\left(f\right),
\end{equation*}
and is closely related, through the Bochner formula, to a lower bound on the Ricci curvature. The Bakry-\'Emery framework was developed for
locally finite graphs by Lin-Yau \cite{Lin-Yau} in 2010, who studied lower Ricci curvature bounds in terms of the $\Gamma$-calculus and obtained eigenvalue estimates. This curvature condition has been studied extensively in \cite{Chung-Lin-Yau, Hua-Lin, Jost-Liu}. In 2015, Bauer et al. \cite{Bauer} introduced the exponential curvature dimension condition $CDE$ and proved the Li-Yau type gradient estimate
\begin{equation}
  \frac{\Gamma\left(\sqrt{u}\right)}{u}-\frac{\partial_{t}u}{2u}\leq \frac{n}{2t}, 
\label{tu li-yau 1}
\end{equation}
for positive solutions of the heat equation on finite graphs satisfying $CDE(n,0)$. In 2018, M\"unch \cite{Munch} introduced the $CD\psi$ condition and the corresponding $\psi$-operators, where $\psi\in C^1(0,+\infty)$ is concave. The $CD\psi$ condition is closely related to the classical $CD$ condition. More precisely, if
$\psi\in C^2(0,+\infty)$ is concave with
$\psi'(1)\neq0$ and $\psi''(1)<0$, then
\begin{equation*}
    CD\psi(n,0)\Rightarrow CD\left(-\frac{\psi^{\prime\prime}(1)}{\left[\psi^{\prime}(1)\right]^{2}}n,0\right). 
\end{equation*}
In particular, both the $CD{\log}(n,0)$ condition and the $CD{\sqrt{\cdot}}(n,0)$ condition imply a classical curvature dimension condition. On finite graphs satisfying $CD\psi(n,0)$, M\"unch \cite{Munch} obtained
\begin{equation}
    \Gamma^{\psi}\left(u\right)-{\psi}^{\prime}\left(1\right)\frac{\partial_{t}u}{u}\leq \frac{n}{2t},
\label{tu li-yau 2}
\end{equation}
for positive solutions of the heat equation. For
$\psi(s)=\sqrt{s}$, \eqref{tu li-yau 2} reduces to
\eqref{tu li-yau 1}, while for $\psi(s)=\log s$ it takes the same form as the classical Li-Yau estimate \eqref{li-yau classical}. Thus, the $CD\psi$ framework provides a unified setting for
logarithmic and nonlogarithmic gradient estimates on graphs. Subsequent work \cite{ Li-Zhang,  Lv-Wang, Wang-Zhang} further investigated gradient estimates under negative curvature bounds within the $CD\psi$ framework, including discrete analogs of the Bakry-Qian type estimate 
\begin{equation}
    \Gamma^\psi(u)-\psi'(1)\left(1+\frac{2Kt}{3}\right)\frac{\partial_tu}{u}\leq\frac{n}{2t}+\frac{nK}{2}\left(1+\frac{Kt}{3}\right), 
\label{Bakry-Qian graph}
\end{equation}
for positive solutions of the heat equation on locally finite weighted graphs satisfying the $CD\psi(n,-K)$ condition with $K\ge0$. In the same setting, the Li-Xu type estimate takes the form
\begin{equation}
    \Gamma^\psi(u)
    -
    \psi'(1)\left(1+\coth(Kt)-Kt\,\operatorname{csch}^2(Kt)\right)\frac{\partial_tu}{u}
    \leq
    \frac{nK}{2}\bigl(\coth(Kt)+1\bigr).
\label{Li-Xu graph}
\end{equation}

Another fundamental geometric property is the volume doubling property. Let $M$ be a complete $n$-dimensional Riemannian manifold satisfying $ \operatorname{Ric}\geq -(n-1)K$ with $K\ge0$. The Bishop-Gromov volume comparison theorem yields a
curvature dependent doubling estimate of the form
\begin{equation}
    V(x,2r)
    \leq
    C(n) e^{C(n)\sqrt{K}r}V(x,r).
\label{manifold exponential doubling}
\end{equation}
In particular, when $K=0$, one obtains a uniform volume doubling and hence a polynomial volume growth. Volume growth is also closely related to spectral geometry. On complete noncompact Riemannian manifolds, Brooks \cite{Brooks} in 1981 established a fundamental relation between exponential volume growth and the bottom of the spectrum of the Laplace-Beltrami operator; see also the monograph of Grigor'yan \cite{Grigoryan} from 2009 for a systematic treatment of
spectral and heat kernel aspects of volume growth. From an analytic point of view, Baudoin-Garofalo \cite{Baudoin-Garofalo} in 2011 gave a heat semigroup proof of volume doubling on complete
Riemannian manifolds with nonnegative Ricci curvature. Their argument is based on a logarithmic Sobolev type inequality and the resulting exponential integrability estimate, combined with the Li-Yau Harnack inequality. This semigroup approach is particularly well suited for adaptation to discrete settings.

On graphs, relations among volume growth, isoperimetry and the spectrum of the Laplacian have also been studied extensively. In
1984, Dodziuk \cite{Dodziuk} established a discrete analogue of Cheeger's estimate relating an isoperimetric constant to the bottom of the spectrum. For weighted graphs, Folz \cite{Folz} in 2014 obtained estimates relating exponential and subexponential volume growth to the bottom of the essential spectrum of general graph Laplacians. These results illustrate that volume growth has not only geometric but also spectral consequences.
Deriving volume growth directly from curvature is, however, more delicate on graphs, since there is no direct analogue of the Bishop-Gromov comparison theorem. In 2015, Bauer et al. \cite{Bauer} proved polynomial volume growth under $CDE(n,0)$. In
2019, Horn et al. \cite{Horn}, following the heat semigroup strategy of Baudoin-Garofalo, established a variational inequality and an
exponential integrability estimate under the stronger condition $CDE^\prime(n,0)$ and derived the volume doubling property. Also in the
classical Bakry-\'Emery setting, M\"unch \cite{Munch 2} introduced a modified nonlinear heat equation and in 2019 used it to prove volume doubling for finite graphs satisfying $CD(n,0)$. As a spectral application, he further showed that there are no expander families satisfying the finite dimensional nonnegative curvature dimension condition. More recently, Russ-Pajot \cite{Russ-Pajot} in 2025 extended M\"unch's modified nonlinear heat equation approach to finite and
infinite weighted graphs of bounded geometry satisfying a classical Bakry-\'Emery curvature-dimension condition. In particular, a key part of their argument is the well-posedness of the modified nonlinear heat equation in the infinite weighted graph setting. In 2026, Guo-Huang-Huang \cite{Guo} further developed this
modified heat flow method and proved that every connected simple graph of bounded degree satisfying the dimension free condition
$CD(\infty,0)$ for the unnormalized Laplacian is volume doubling and supports a scale invariant local $L^2$-Poincar\'e inequality.

The results described above concerning curvature and volume doubling are mainly restricted to nonnegative curvature conditions. Although gradient estimates and Harnack inequalities are available for locally finite weighted graphs satisfying the $CD\psi(n,-K)$ condition with $K\geq0$, the corresponding consequences for volume growth under
negative $CD\psi$ curvature bounds remain less understood. In this paper, we develop a heat semigroup approach to both gradient estimates and volume growth under the $CD\psi$ condition. We first establish a general variational inequality for the heat semigroup. By choosing suitable auxiliary functions, we derive from it a family
of Li-Yau type gradient estimates containing
\eqref{tu li-yau 2}, \eqref{Bakry-Qian graph} and \eqref{Li-Xu graph} as special cases. In particular, the Li-Xu type estimate yields a heat kernel
Harnack inequality under negative curvature bounds. We then adapt the heat semigroup approach to volume doubling to the nonlinear $CD\psi$ setting. Under suitable assumptions on $\psi$,
we establish an exponential integrability estimate and derive from it a curvature dependent heat retention estimate for metric balls. Combining the heat retention estimate with the heat kernel Harnack inequality, we obtain a curvature dependent exponential volume doubling estimate under $CD\psi(n,-K)$. When $K=0$, this reduces to a uniform volume doubling. As a spectral consequence, we further show that on an
infinite graph satisfying the same assumptions, the bottom of the spectrum of $-\Delta$ is zero.

Our first main result is a family of Li-Yau type gradient estimates with negative $CD\psi$ curvature bounds. Although a more general estimate involving an arbitrary auxiliary function $W$ is established in Section 4, the following explicit family illustrates its main consequences.
\begin{theorem}
    Let $G=(V,E)$ be a locally finite weighted graph satisfying the $CD\psi\left(n,-K\right)$ condition with $K\geq0$, where $\psi\in C^{1}\left(0,+\infty\right)$ is a concave function. Let $b>\frac{1}{2}$ be a constant. Then for any $0<f\in \ell^{\infty}\left(V\right)$,
    \begin{equation}
        \Gamma ^\psi\left(P_t f\right)\leq \psi^\prime \left(1\right)\left(1+\frac{2Kt}{2b+1}\right)  \frac{\Delta P_t f}{P_t f}+\frac{n}{2} \left(\frac{b^2}{\left(2b-1\right)t}+\frac{K^2 t}{2b+1}+K\right).
    \label{intro gradient family}
    \end{equation} 
\end{theorem}
Taking $b=1$ in \eqref{intro gradient family} recovers the Bakry-Qian type estimate \eqref{Bakry-Qian graph}, while setting $K=0$ further yields Münch's estimate \eqref{tu li-yau 2}. We next turn to the volume growth consequences of the $CD\psi(n,-K)$ condition with $K\ge0$. Recall that $H_\psi$ and $\eta_\psi$ are defined in Section 2.

\begin{theorem}
    Let $G=(V,E)$ be a locally finite weighted graph satisfying the $CD\psi\left(n,-K\right)$ condition with $K\geq0$, where $\psi\in C^{1}\left(0,+\infty\right)$ is a concave function with $\psi^\prime \left(1\right)>0,\ H_{\psi}<\infty$ and $\limsup _{\lambda \downarrow 0}\frac{\eta_{\psi}(\lambda)}{\lambda^{2}}<\infty $. Then there exist positive constants $c,C$ depending only on $n,\psi,D_\mu,\mu_{\max}, w_{\min}$ such that for any $x\in V$ and any $r>0$,
    \begin{equation}
        V\left(x,2r\right)\leq Ce^{c\sqrt{K}r}V\left(x,r\right).
    \label{intro negative doubling}
    \end{equation}
\end{theorem}
The curvature radius dependence $\sqrt K\,r$ agrees with the scale appearing in the classical Bishop--Gromov comparison estimate \eqref{manifold exponential doubling}. The linear dependence on the radius in the exponential factor
cannot, in general, be replaced by a sublinear one. When $K=0$, \eqref{intro negative doubling} gives a uniform volume doubling property. In the infinite graph case, this further implies the vanishing of the bottom of the spectrum of $-\Delta$.

The remainder of this paper is organized as follows. In Section 2, we introduce the basic notation for locally finite weighted graphs,
recall the $CD\psi$ condition, and collect the properties of the heat kernel and the heat semigroup used throughout the paper. In Section 3, we establish the main variational inequality for the heat semigroup. In Section 4, by choosing suitable auxiliary functions in this variational inequality, we derive a family of Li-Yau type
gradient estimates and a heat kernel Harnack inequality under negative curvature bounds. In Section 5, we establish an exponential integrability estimate and derive a curvature-dependent heat retention estimate for metric balls. Combining this estimate with the Harnack inequality
obtained in Section 4, we prove the curvature dependent exponential volume doubling property under $CD\psi(n,-K)$ with $K\ge0$. We conclude with a spectral consequence of the uniform volume doubling property in the case $K=0$.

\section{Preliminaries}
Let $G=(V,E)$ be a graph, where $V$ denotes the vertex set and $E$ denotes the edge set of $G$. For $x,y\in V$, we write $x\sim y$ if $x$ and $y$ are adjacent, i.e., there is an edge connecting $x$ and $y$. The graph $G$ is called connected if for any $x,y\in V$, there exists a finite path $x=x_0\sim x_1\sim \cdots \sim x_n=y$; the length of such a path is defined as $n$.

Given a measure $\mu :V\rightarrow (0,+\infty)$ on the vertices and a weight function $w:E\rightarrow (0,+\infty)$ on edges such that for any adjacent vertices $x,y$, the weight satisfies $w_{xy}>0$ and $w_{xy}=w_{yx}$, the quadruple $G=(V,E,\mu,w)$ is called a weighted graph. To simplify notation, we always regard $G=(V,E)$ as a connected weighted graph with the measure $\mu$ and the weight $w$ implicitly understood, unless otherwise stated.

A graph $G$ is called locally finite if for every $x\in V$ the number of vertices adjacent to $x$ is finite. In this paper, we further require locally finite weighted graphs to satisfy the following conditions
\begin{equation*}
    D_{\mu}:=\sup_{x\in V}\frac{\deg(x)}{\mu(x)}<\infty,\ \ \mu_{\max}:=\sup_{x\in V}\mu(x)<\infty,\ \  w_{\min}:=\inf_{x\sim y}w_{xy}>0,
\end{equation*}
where 
$\text{deg}(x):=\sum_{y\sim x}w_{xy}$ denotes the weighted degree of $x$.

For $x,y\in V$, the distance $d(x,y)$ is defined as the minimum of the lengths of all paths connecting $x$ and $y$. The ball $B(x,r)$ on the graph is defined by
\begin{equation*}
    B(x,r):=\left\{y\in V\mid d(x,y)<r\right\},
\end{equation*}
and the volume $V(x,r)$ of the ball $B(x,r)$ is defined by
\begin{equation*}
V(x,r):=\sum_{y\in B(x,r)}\mu(y).
\end{equation*}

We denote by $\mathbb{R}^{V}$ and $\ell^{\infty}(V)$ the spaces of all real-valued functions and all bounded functions on the graph, respectively. The $\ell^{\infty}$-norm of $f\in \ell^{\infty}(V)$ is defined by
\begin{equation*}
    \left \| f\right \|_{\infty}:=\sup_{x\in V}|f(x)|.     
\end{equation*}
A function $f\in \mathbb{R}^{V}$ is called $d$-Lipschitz if for any adjacent vertices $x\sim y$ one has
\begin{equation*}
    \left | f(y)-f(x)\right |\le d. 
\end{equation*}
For a strictly positive function $0<f\in\mathbb{R}^{V}$, we define the constant $R_{f}$ as follows
\begin{equation*}
    R_{f}:=\sup_{y\sim x}\frac{f(y)}{f(x)}. 
\end{equation*}

\subsection{Curvature dimension condition}
For any $f\in \mathbb{R}^{V}$ and any $x\in V$, the Laplace operator $\Delta$ is defined by 
\begin{equation*}
    \Delta f\left(x\right):=\frac{1}{\mu\left(x\right)}\sum_{y\sim x}w_{xy}\left[f\left(y\right)-f\left(x\right)\right]. 
\end{equation*}

Let $\psi\in C^1(0,+\infty)$. For any $0<f\in \mathbb{R}^{V}$ and any $x\in V$, the $\psi$-Laplace operator is defined by
\begin{equation*}
    \left(\Delta^{\psi}f\right)\left(x\right):=\Delta\left[\psi\left(\frac{f}{f\left(x\right)}\right)\right]\left(x\right).
\end{equation*}
We first introduce the function $\overline{\psi}$. For a concave function $\psi\in C^{1}(0,+\infty)$, set
\begin{equation*}
    \overline{\psi}\left(s\right):=\psi^{\prime}\left(1\right)\left(s-1\right)-\left[\psi\left(s\right)-\psi\left(1\right)\right].
\end{equation*}
Since $\psi$ is concave, $\overline{\psi}(s)\ge 0$ for all $s>0$, and $\overline{\psi}(1)=0$. The $\psi$-gradient operator is defined as
\begin{equation*}
    \Gamma^{\psi}:=\Delta^{\overline{\psi}}.
\end{equation*}
By the properties of $\overline{\psi}$, for $0<f\in \mathbb{R}^{V}$ and $ x_0\in V$, the function 
\begin{equation*}
    x\mapsto\overline{\psi}\left(\frac{f(x)}{f\left(x_{0}\right)}\right)
\end{equation*}
reaches its minimum at $x=x_0$, which implies $\Gamma^{\psi}\ge0$. The following lemma relates $\Delta^{\psi}$ and $\Gamma^{\psi}$.
\begin{lemma}
    Let $G=(V,E)$ be a locally finite weighted graph and let $\psi\in C^{1}(0,+\infty)$. Then for any function $0<f\in\mathbb{R}^{V}$, 
\begin{equation*}
    -\Delta^{\psi}f=\Gamma^{\psi}\left(f\right)-{\psi}^{\prime}\left(1\right)\frac{\Delta f}{f}.
\end{equation*}
\label{chain psi}
\end{lemma}

\begin{proof}
By the definition of $\Gamma^{\psi}$,  we have
\begin{equation*}
    \Gamma^{\psi}f=\Delta^{\overline\psi}f=\psi^{\prime}\left(1\right)\Delta^{id}f-\Delta^{\psi}f=\psi^{\prime}(1)\frac{\Delta f}{f}-\Delta^{\psi}f.
\end{equation*}
Rearranging gives the desired identity.
\end{proof}
The second $\psi$-gradient form is defined for any $0<f\in \mathbb{R}^{V}$ by
\begin{equation*}
    2\Gamma_{2}^{\psi}\left(f\right):=\Omega^{\psi}\left(f\right)+\frac{\Delta f\cdotp\Delta^{\psi}f}{f}-\frac{\Delta \left(f\cdotp\Delta^{\psi}f\right)}{f},
\end{equation*}
where, for any $x\in V$,
\begin{equation*}
    \Omega^{\psi}f\left(x\right):=\Delta\left[\psi^{\prime}\left(\frac{f}{f\left(x\right)}\right)\cdotp\frac{f}{f\left(x\right)}\left(\frac{\Delta f}{f}-\frac{\Delta f\left(x\right)}{f\left(x\right)}\right)\right]\left(x\right).
\end{equation*}
We now recall the $\psi$-curvature-dimension condition (the $CD\psi(n,K)$ condition) on graphs. A graph satisfies the $CD\psi(n,K)$ condition if for every $0<f\in \mathbb{R}^{V}$,
\begin{equation*}
    \Gamma_2^{\psi}\left(f\right)\geq \frac{1}{n}\left(\Delta^{\psi} f\right)^2+K\Gamma^{\psi}\left(f\right).
\end{equation*}

The Harnack constant $H_{\psi}$ is defined by
\begin{equation*}
    H_{\psi}:=\sup_{s>1}\frac{\left(\log s\right)^2}{\overline{\psi}(s)}.
\end{equation*}
For $\lambda\ge0$, we define $\eta_{\psi}(\lambda)$ by
\begin{equation*}
    \eta_{\psi}(\lambda):=\max\left\{\overline{\psi}\left(e^{\lambda }\right),\overline{\psi}\left(e^{-\lambda }\right)\right\}. 
\end{equation*}

\subsection{Heat kernel and heat semigroup}
For every $f\in \ell^{\infty}(V)$, we have 
\begin{equation*}
    \left \| \Delta f\right . \|_{\infty}\leq2D_{\mu}\left \| f\right \|_{\infty}. 
\end{equation*}
Hence, $\Delta$ is a bounded operator on $\ell^{\infty}(V)$, and the semigroup generated by $\Delta$ is defined for $t\ge0$ by
\begin{equation*}
    P_{t}:=e^{t\Delta}=\sum_{k=0}^{\infty}\frac{t^{k}}{k!}\Delta^{k}. 
\end{equation*}
The heat kernel $p:[0,+\infty)\times V\times V\rightarrow[0,+\infty)$ associated with $\Delta$ is defined by
\begin{equation*}
    p(t,x,y)=\frac{1}{\mu(y)}P_{t}\textbf{1}_{\left\{y\right\}}(x). 
\end{equation*}
Then $p$ is the integral kernel of $P_t$ with respect to $\mu$, namely, for any $f\in \ell^{\infty}(V)$
\begin{equation}
    P_tf(x)=\sum_{y\in V}\mu(y)p(t,x,y)f(y).
\label{semigroup and kernel}
\end{equation}

\begin{proposition}
    Let $G=(V,E)$ be a locally finite weighted graph. For all $x,y\in V$  and $f,g\in \ell^{\infty}(V)$, the following properties hold 
    \begin{enumerate}
        \item $P_{0}=I, \ \ P_{t}P_{s}=P_{t+s}$.
        \item $f\ge 0 $ implies $P_{t}f\ge0$, and $f\ge g $ implies $P_{t}f\ge P_{t}g$. 
        \item $P_{t}\textbf{1}=\textbf{1},\ \ \left \| P_t f\right \|_{\infty}\le\left \|f \right \|_{\infty}$.
        \item $\partial_{t}P_{t}f=\Delta P_{t}f=P_{t}\Delta f$ for $t>0$.
        \item $p(t,x,y)>0$ for $ t>0$.
        \item $p(t,x,y)=p(t,y,x)$.
        \item $p(t+s,x,y)=\sum_{z\in V}\mu(z)p(t,x,z)p(s,z,y)$.
        \item $\sum_{z\in V}\mu(z)p(t,x,z)=1$. 
        \item $\partial_{t}p(t,x,y)=\Delta_{x}p(t,x,y)=\Delta_{y}p(t,x,y)$ for $ t>0$.
    \end{enumerate}
\label{heat}
\end{proposition}

\begin{proof}
    These properties follow from the boundedness of $\Delta$, the Markov property of $(P_t)_{t\geq0}$, and the symmetry of the edge weights; see, e.g., \cite{Horn}. 
\end{proof}

\section{The main variational inequality}

Fix $T>0$. For any $0<f\in \ell^{\infty}\left(V\right)$, we define $v(t,x)$ by 
\begin{equation}
    v\left(t,x\right):=P_{T-t}f\left(x\right),
\label{v def}
\end{equation}
where $t\in [0,T]$ and $x\in V$. 

\begin{lemma}
     Let $G=\left(V,E\right)$ be a locally finite weighted graph and let $\psi\in C^{1}\left(0,+\infty\right)$ be a concave function. Then for every $0<f\in \ell^{\infty}\left(V\right)$ with $R_{f}<\infty$, the functions $\Delta^{\psi}v, v\cdot \Gamma^{\psi}(v)$ and $ v\cdot \Omega^{\psi}(v)$ are uniformly bounded on $[0,T]\times V$, where $v$ is given by \eqref{v def}.
\label{uniform}
\end{lemma}
\begin{proof}
    For  $x\in V$, we have
    \begin{equation}
    \begin{aligned}
        \left|\Delta f(x)\right|\le f(x)\frac{1}{\mu(x)}\sum_{y\sim x}\omega_{xy}
        \left|\frac{f(y)}{f(x)}-1\right|\leq D_\mu (R_{f}+1) f(x).
    \label{laplace bounded}
    \end{aligned}
    \end{equation}
    Set $a:=D_{\mu}(R_{f}+1)$. Then $\left|\Delta f\right|\le af$. Using properties 2 and 4 of Proposition \ref{heat}, we obtain
    \begin{equation*}
        \partial_t P_t f=P_t\Delta f\le aP_tf.
    \end{equation*}
    Consequently, $P_t f\le e^{at}f$. A similar argument with $\Delta f\ge-a f$ gives $P_t f\ge e^{-at}f$. Now fix $x\in V$ and $t\in [0,T]$. For any $y\sim x$, 
    \begin{equation}
        R_f^{-1} e^{-2aT}\le\frac{v(y)}{v(x)}=\frac{P_{T-t} f(y)}{P_{T-t} f(x)} \le R_f e^{2aT},
    \label{v sim}
    \end{equation}
    Since $\psi$ is continuous on $(0,+\infty)$, the ratio $\frac{v(y)}{v(x)}$ remains in a compact set, and therefore there exists a constant $C_{1}=C_{1}\left(\psi,f,T\right)$ such that 
    \begin{equation*}
        \left| \psi\left(\frac{v(y)}{v(x)}\right)\right|\leq C_1. 
    \end{equation*}
    Hence, we obtain 
    \begin{equation*}
    \begin{aligned}
        \left|\Delta^{\psi}v(x)\right|&=\left|\Delta\psi\left(\frac{v(y)}{v(x)}\right)\right|=\left|\frac{1}{\mu(x)}\sum_{y\sim x}w_{xy}\left[ \psi\left(\frac{v(y)}{v(x)}\right)-\psi(1)\right]\right|\\
        &\leq\frac{1}{\mu(x)}\sum_{y\sim x}w_{xy}\left| \psi\left(\frac{v(y)}{v(x)}\right)-\psi(1)\right|
        \leq \left[C_1+\left|\psi(1)\right|\right] D_{\mu}. 
    \end{aligned}
    \end{equation*}
    Moreover, using \eqref{v sim} together with the boundedness of $v$ (since $\left \| v\right \|_{\infty}\le\left \|f \right \|_{\infty}$), an estimate similar to \eqref{laplace bounded} yields a constant $C_{2}=C_{2}\left(f,D_{\mu}\right)$ such that $\left|\Delta v\right|\leq C_{2}$. 
    By Lemma \ref{chain psi}, multiplying by $v$ yields
    \begin{equation*}
        v\cdot\Gamma^{\psi}(v)=\psi^{\prime}(1)\Delta v-v\cdot\Delta^{\psi}v. 
    \end{equation*}
    Thus, $ v\cdot\Gamma^{\psi}(v)$ is uniformly bounded on $[0,T]\times V$. The continuity of $\psi^{\prime}$ on $(0,+\infty)$ guarantees that for $y\sim x$ there exists $C_{3}=C_{3}\left(\psi,f,T\right)$ with
    \begin{equation*}
        \left| \psi^{\prime}\left(\frac{v(y)}{v(x)}\right)\right|\leq C_3. 
    \end{equation*}
    Finally, for the $\Omega^{\psi}$ term we compute 
    \begin{equation*}
    \begin{aligned}
        \left(v\Omega^{\psi}v\right)\left(x\right)&=v(x)\Delta\left[\psi^{\prime}\left(\frac{v}{v\left(x\right)}\right)\cdotp\frac{v}{v\left(x\right)}\left(\frac{\Delta v}{v}-\frac{\Delta v\left(x\right)}{v\left(x\right)}\right)\right]\left(x\right)\\
        &=\frac{1}{\mu(x)}\sum_{y\sim x}w_{xy}\psi^{\prime}\left(\frac{v(y)}{v\left(x\right)}\right)\left[\Delta v(y)-\Delta v\left(x\right)\frac{v(y)}{v\left(x\right)}\right], 
    \end{aligned}  
    \end{equation*}
    where we used the fact that the function inside $\Delta$ vanishes at $x$. Hence,    
    \begin{equation*}   
    \begin{aligned}
        \left|\left(v\Omega^{\psi}v\right)\left(x\right)\right|
        &\leq\frac{1}{\mu(x)}\sum_{y\sim x}w_{xy}\left|\psi^{\prime}\left(\frac{v(y)}{v\left(x\right)}\right)\right|\left( \left|\Delta v(y)\right|+\left|\Delta v\left(x\right)\right|\frac{v(y)}{v\left(x\right)}\right)\\
        &\leq C_2C_{3} D_{\mu}\left(1+R_f e^{2aT}\right).
    \end{aligned}
    \end{equation*}
\end{proof}

By Lemma \ref{uniform}, $v\cdotp \Delta^{\psi}v$ is uniformly bounded on $[0,T]\times V$. For $t\in [0,T]$ and $x\in V$,  define $\phi(t,x)$ by
\begin{equation}
   \phi\left(t,x\right):= P_t\left(v\cdotp \Gamma^{\psi}\left(v\right)\right)\left(x\right), 
\label{phi def}
\end{equation}
where $v$ is given by \eqref{v def}. 
\begin{lemma}
     Let $G=\left(V,E\right)$ be a locally finite weighted graph and let $\psi\in C^{1}\left(0,+\infty\right)$ be a concave function. Then for every $0<f\in \ell^{\infty}\left(V\right)$ with $R_{f}<\infty$, 
     \begin{equation*}
         \partial_t \phi\left(t,x\right)=2P_t\left(v\cdotp\Gamma_2^{\psi}\left(v\right)\right)\left(x\right), \ \ (t,x)\in[0,T]\times V,
     \end{equation*}
    where $v$ and $\phi$ are given by \eqref{v def} and \eqref{phi def}, respectively.
\label{partial}
\end{lemma}

\begin{proof}
    By Proposition \ref{heat} (4) we have 
    \begin{equation}
         P_t\left(\Delta v\right)=\Delta\left(P_t v\right)=\Delta\left(P_{T}f\right).
    \label{delta P v}
    \end{equation}
    Using Lemma \ref{chain psi} to replace $\Gamma^{\psi}\left(v\right)$, we obtain
    \begin{equation*}
    \begin{aligned}
         \phi\left(t,x\right)&=-P_t\left(v\cdotp \Delta^{\psi}v\right)\left(x\right)+\psi^{\prime}\left(1\right)P_t\Delta v\left(x\right)\\
         &=-P_t\left(v\cdotp \Delta^{\psi}v\right)\left(x\right)+\psi^{\prime}\left(1\right)\Delta\left(P_{T}f\right)\left(x\right).
    \end{aligned}
    \end{equation*}
    Since the second term is independent of $t$, we have
    \begin{equation*}
        \partial_t \phi\left(t,x\right)=-\partial_{t}P_t\left(v\cdotp \Delta^{\psi}v\right)\left(x\right). 
    \end{equation*}
    In terms of the heat kernel, this semigroup can be written as 
    \begin{equation*}
        P_t\left(v\cdotp \Delta^{\psi}v\right)\left(x\right)= \sum_{y\in V}\mu\left(y\right)p\left(t,x,y\right)\left(v\cdotp \Delta^{\psi}v\right)\left(y\right). 
    \end{equation*}
    To differentiate termwise we also need uniform convergence of the termwise differentiated series. From Proposition \ref{heat} (9), the termwise derivative is
    \begin{equation*}
        \sum_{y\in V}\mu\left(y\right)\left[\Delta_{x}p\left(t,x,y\right)\left(v\cdotp \Delta^{\psi}v\right)\left(y\right)+p\left(t,x,y\right)\partial_{t}\left(v\cdotp \Delta^{\psi}v\right)\left(y\right)\right]. 
    \end{equation*}
    Since $v\cdotp \Delta^{\psi}v$ is uniformly bounded on $[0,T]\times V$, we may interchange the Laplacian with the summation to obtain
    \begin{equation*}
    \begin{aligned}
        &\sum_{y\in V}\mu(y)\Delta_{x}p(t,x,y)\left(v\cdotp \Delta^{\psi}v\right)(y)\\&=\Delta\left[\sum_{y\in V}\mu(y)p(t,x,y)\left(v\cdotp \Delta^{\psi}v\right)(y)\right]=\Delta P_{t}\left(v\cdotp \Delta^{\psi}v\right)(x)\\&=P_{t}\Delta \left(v\cdotp \Delta^{\psi}v\right)(x)=\sum_{y\in V}\mu(y)p(t,x,y)\Delta \left(v\cdotp \Delta^{\psi}v\right)(y).
    \end{aligned} 
    \end{equation*}
    where we also used the commutativity of $P_{t}$ and $\Delta$ from Proposition \ref{heat}. From the definition of $\Delta^{\psi}$, we compute
    \begin{equation*}
    \begin{aligned}
        \partial_t\left(\Delta^{\psi}v\right)\left(y\right)&=\Delta\partial_t\left[\psi \left(\frac{v}{v\left(y\right)}\right)\right]\left(y\right)\\
        &=\Delta\left[\psi^{\prime}\left(\frac{v}{v\left(y\right)}\right)\cdotp\frac{v}{v\left(y\right)}\cdotp\left(\frac{-\Delta v}{v}+\frac{\Delta v\left(y\right)}{v\left(y\right)}\right)
        \right]\left(y\right)\\
        &=-\left(\Omega^{\psi}v\right)\left(y\right).
    \end{aligned}
    \end{equation*}
    Thus, we conclude that $ \partial_t\left( \Delta^{\psi}v\right)=-\Omega^{\psi}v$. Using
    \begin{equation*}
        \partial_t\left(v\cdotp \Delta^{\psi}v\right)=-\Delta v\cdotp \Delta^{\psi}v-v\cdotp\Omega^{\psi}v,
    \end{equation*}
    and the boundedness of $\Delta v, \Delta^{\psi}v$ and $v\Omega^{\psi}v$, the differentiated series also converges uniformly in $(t,x)$. Hence, differentiation and summation can be exchanged, giving
    \begin{equation*}
        \partial_{t}P_t\left(v\cdotp \Delta^{\psi}v\right)\left(x\right)=\sum_{y\in V}\mu\left(y\right)p\left(t,x,y\right)\left[\Delta\left(v\cdotp \Delta^{\psi}v\right)-\Delta v\cdotp \Delta^{\psi}v-v\cdotp\Omega^{\psi}v\right]\left(y\right),
    \end{equation*}
    Consequently, we have
    \begin{equation*}
        \partial_t \phi\left(t,x\right)=-P_{t}\left(\Delta\left(v\cdotp \Delta^{\psi}v\right)-\Delta v\cdotp \Delta^{\psi}v-v\cdotp\Omega^{\psi}v\right)(x)=2P_t\left(v\cdotp\Gamma_2^{\psi}v\right)\left(x\right),
    \end{equation*}
    where the last equality follows from the definition of $\Gamma_{2}^{\psi}$, namely,
    \begin{equation*}
        2v\cdotp\Gamma_2^{\psi}v=-\Delta\left(v\cdotp \Delta^{\psi}v\right)+\Delta v\cdotp \Delta^{\psi}v+v\cdotp\Omega^{\psi}v. 
    \end{equation*}
    Recalling the definition of $v$, we obtain the desired identity. The uniform boundedness used above follows from Lemma \ref{uniform}, which guarantees the uniform convergence of both series and hence justifies the interchange; this completes the proof.
\end{proof}

\begin{theorem}
    Let $G=\left(V,E\right)$ be a locally finite weighted graph satisfying the $CD\psi\left(n,-K\right)$ condition with $K\geq0$, where $\psi\in C^{1}\left(0,+\infty\right)$ is a concave function. Then for any $0<f\in \ell^{\infty}\left(V\right)$ with $R_{f}<\infty$, 
    \begin{equation}
         \partial_t\left(\alpha \phi\right)\geq\left(\alpha^{\prime}-\frac{4\alpha\gamma}{n}-2\alpha K\right)\phi+\frac{4\psi^{\prime}\left(1\right)\alpha\gamma}{n}\Delta P_T f-\frac{2\alpha\gamma^2}{n} P_T f,  
    \label{variational ineq}
    \end{equation}
     where $\phi$ is defined as in \eqref{phi def}, $\alpha\in C^{1}\left[0,T\right]$ is positive,  and $\gamma\in C\left[0,T\right]$.
\label{variational th}
\end{theorem}

\begin{proof}
    By Lemma \ref{partial}, we obtain
    \begin{equation*}
    \begin{aligned}
        \partial_t\left(\alpha\phi\right)\left(x\right)&=\alpha^{\prime}\phi\left(x\right)+2\alpha P_t\left[v\cdotp\Gamma_2^{\psi}v\right]\left(x\right)\\
        &\geq \alpha^{\prime}\phi\left(x\right)+2\alpha P_t\left[\frac{1}{n}v\left(\Delta^{\psi}v\right)^{2}-Kv\cdotp \Gamma^{\psi}v\right]\left(x\right)\\
        &=\left(\alpha^{\prime}-2\alpha K\right)\phi\left(x\right)+\frac{2\alpha}{n}\sum_{y\in V}\mu\left(y\right)p\left(t,x,y\right)v\left(y\right)\left(\Delta^{\psi}v\right)^{2}\left(y\right),
    \end{aligned}
    \end{equation*}
    In the second line, we used the $CD\psi(n,-K)$ condition. For the quadratic term we apply the elementary inequality $a^{2}\ge2\gamma a-\gamma^{2}$ (with $a=\Delta^{\psi}v$ and multiplying by $v>0$) to get 
    \begin{equation*}
        v\left(y\right)\left(\Delta^{\psi}v\right)^{2}\left(y\right)\geq2\gamma v\left(y\right)\cdotp\Delta^{\psi}v\left(y\right)-\gamma^2 v\left(y\right). 
    \end{equation*}
    By Lemma \ref{chain psi}, $v\cdotp\Delta^{\psi}v=\psi^{\prime}\left(1\right)\Delta v-v\cdotp\Gamma^{\psi}v$. Substituting this into the right-hand side yields
    \begin{equation*}
        v\left(y\right)\left(\Delta^{\psi}v\right)^{2}\left(y\right)\geq2\gamma\big{[}\psi^{\prime}\left(1\right)\Delta v\left(y\right)-v\left(y\right)\cdotp\Gamma^{\psi}v\left(y\right)\big{]}-\gamma^2 v\left(y\right). 
    \end{equation*}
    Inserting this estimate into the sum, we obtain
    \begin{equation*}
    \begin{aligned}
        &\sum_{y\in V}\mu\left(y\right)p\left(t,x,y\right)v\left(y\right)\left(\Delta^{\psi}v\right)^{2}\left(y\right)\\
        &\ge 2\psi^{\prime}(1)\gamma\Delta P_{T}f(x)-2\gamma\phi(x)-\gamma^{2}P_{T}f(x). 
    \end{aligned}
    \end{equation*}
    Here we have used \eqref{delta P v} together with the definition of $\phi$. Multiplying this inequality by $\frac{2\alpha}{n}$ yields the desired estimate, which completes the proof.

\end{proof}

\section{A family of Li-Yau inequalities}

\begin{theorem}
   Let $G=\left(V,E\right)$ be a locally finite weighted graph satisfying the $CD\psi\left(n,-K\right)$ condition with $K\geq0$, where $\psi\in C^{1}\left(0,+\infty\right)$ is a concave function. Assume that $W\in C^{1}[0,T)\cap C[0,T]$ satisfies $W\left(0\right)=1, W\left(T\right)=0, W>0 \ on\ [0,T)$ and $W^{\prime}\in L^{2}(0,T)$. Then for any $0<f\in \ell^{\infty}\left(V\right)$, 
    \begin{equation*}
    \begin{aligned}
        \Gamma^{\psi}\left(P_T f\right)&\leq\psi^{\prime}\left(1\right)\left(1+2K\int_{0}^{T}W\left(s\right)^2ds\right)\frac{\Delta P_T f}{P_T f}\\
        &+\frac{n}{2}\left(\int_{0}^{T}W^\prime\left(s\right)^2ds+K^2\int_{0}^{T}W\left(s\right)^2ds+K\right).
    \end{aligned} 
    \end{equation*}
\label{zu}
\end{theorem}

\begin{proof}
    We first assume that $f$ satisfies $R_{f}<\infty$. The general case will be handled by an approximation argument at the end. Fix $\delta\in(0,T)$ and work on the interval $[0,T-\delta]$. Choose the parameters $\alpha$ and $\gamma$ in the variational inequality of Theorem \ref{variational th} as follows. Set
    \begin{equation*}
        \gamma=\frac{n}{4}\left(\frac{\alpha^{\prime}}{\alpha}-2K\right),
    \end{equation*}
    which guarantees
    \begin{equation*}
        \alpha^{\prime}-\frac{4\alpha\gamma}{n}-2\alpha K=0.
    \end{equation*}
    Set $\alpha:=W^2$. Substituting this choice into \eqref{variational ineq}, we obtain
    \begin{equation*} 
    \begin{aligned}
    \partial_t\left(W^2\phi\right)\ge \psi^{\prime}\left(1\right)\left(2WW^{\prime}-2KW^{2}\right)\Delta P_T f\\
    -\frac{n}{2}\left(\left(W^{\prime}\right)^2-2KWW^{\prime}+K^2W^2\right)P_T f.
    \end{aligned}
    \end{equation*}
    Integrate this inequality with respect to $t$ over $[0,T-\delta]$. Using $ \phi\left(0\right)=P_T f\cdotp\Gamma^{\psi}\left(P_T f\right)$ and $W\left(0\right)=1$, we get
    \begin{equation*}
    \begin{aligned}       
        &W\left(T-\delta\right)^{2}\phi\left(T-\delta\right)-P_T f\cdotp\Gamma^{\psi}\left(P_T f\right)\\
        &\geq\psi^{\prime}\left(1\right)\left(W\left(T-\delta\right)^{2}-1-2K\int_{0}^{T-\delta}W\left(s\right)^2ds\right)\Delta P_T f \\
        &-\frac{n}{2}\left[\int_{0}^{T-\delta}W^\prime\left(s\right)^2ds+K^2\int_{0}^{T-\delta}W\left(s\right)^2ds-KW\left(T-\delta\right)^{2}+K\right]P_T f. 
    \end{aligned}
    \end{equation*}
    By Lemma \ref{uniform}, $\phi$ is uniformly bounded on $[0,T]\times V$. Since $W(T)=0$, letting $\delta\rightarrow0$ gives $W\left(T-\delta\right)^{2}\phi\left(T-\delta\right)\rightarrow0$. Passing to the limit yields the desired estimate, which proves the theorem under the additional condition  $R_{f}<\infty$.
    
    Now, let $f$ be an arbitrary positive function in $ \ell^{\infty}\left(V\right)$. For $\epsilon>0$ set $f_\epsilon=f+\epsilon$. Clearly, $R_{f_{\epsilon}}<\infty$. Applying the estimate already established to $f_{\epsilon}$ gives
    \begin{equation}
    \begin{aligned}
        \Gamma^{\psi}\left(P_T f_\epsilon\right)\leq&\psi^{\prime}\left(1\right)\left(1+2K\int_{0}^{T}W\left(s\right)^2ds\right)\frac{\Delta P_T f_\epsilon}{P_T f_\epsilon}\\
        &+\frac{n}{2}\left(\int_{0}^{T}W^\prime\left(s\right)^2ds+K^2\int_{0}^{T}W\left(s\right)^2ds+K\right).
    \label{gradient ep}
    \end{aligned} 
    \end{equation}
    By Proposition \ref{heat} (3), $ P_T f_\epsilon=P_{T}f+\varepsilon\rightarrow P_{T}f$ as $\epsilon\rightarrow 0$. Moreover, from the definitions of $\Delta$ and $\Gamma^{\psi}$, 
    \begin{equation*}
        \Delta P_T f_\epsilon= \Delta P_T f,\ \  \Gamma^{\psi}\left(P_T f_\epsilon\right)\rightarrow\Gamma^{\psi}\left(P_T f\right)\ \  \text{as} \ \epsilon\rightarrow 0.
    \end{equation*}
    Let $\epsilon\rightarrow 0$ in \eqref{gradient ep} therefore yield the desired inequality for $f$, and the proof is complete.
\end{proof}
\begin{remark}
        Taking $W(t)=\frac{\sinh\left[K\left(T-t\right)\right]}{\sinh\left(KT\right)}$ in the family of gradient estimates given in Theorem \ref{zu} recovers the Li-Xu type estimate \eqref{Li-Xu graph}. 
\end{remark}

\begin{theorem}
    Let $G=\left(V,E\right)$ be a locally finite weighted graph satisfying the $CD\psi\left(n,-K\right)$ condition with $K\geq0$, where $\psi\in C^{1}\left(0,+\infty\right)$ is a concave function. Let $b>\frac{1}{2}$ be a constant. Then for any $0<f\in \ell^{\infty}\left(V\right)$, 
    \begin{equation*}
        \Gamma ^\psi\left(P_t f\right)\leq \psi^\prime \left(1\right)\left(1+\frac{2Kt}{2b+1}\right)  \frac{\Delta P_t f}{P_t f}+\frac{n}{2} \left(\frac{b^2}{\left(2b-1\right)t}+\frac{K^2 t}{2b+1}+K\right).
    \end{equation*} 
\label{class li-yau b}
\end{theorem}

\begin{proof}
    We apply Theorem \ref{zu} with the choice
    \begin{equation*}
        W\left(t\right):=\left(1-\frac{t}{T}\right)^{b}.
    \end{equation*}
    A direct computation gives
    \begin{equation*}
        \int_{0}^{T}W\left(s\right)^2ds=\frac{T}{2b+1},\ \ \ \ \ \ \ \int_{0}^{T}W^\prime\left(s\right)^2ds=\frac{b^2}{\left(2b-1\right)T}. 
    \end{equation*}
    Substituting these expressions into the estimate of Theorem \ref{zu} yields, 
    \begin{equation*}
        \Gamma ^\psi\left(P_T f\right)\leq \psi^\prime \left(1\right)\left(1+\frac{2KT}{2b+1}\right)  \frac{\Delta P_T f}{P_T f}+\frac{n}{2} \left(\frac{b^2}{\left(2b-1\right)T}+\frac{K^2 T}{2b+1}+K\right).
    \end{equation*} 
    Since $T$ is arbitrary, we may replace $T$ with $t$ to obtain the estimate claimed.
\end{proof}

\begin{remark}
    The family of gradient estimates given in Theorem \ref{class li-yau b} recovers two previously known inequalities. Taking $b=1$ yields the Bakry-Qian type estimate \eqref{Bakry-Qian graph}. Setting further $K=0$ reduces it to the M\"unch's estimate \eqref{tu li-yau 2}. 
\end{remark}

We now use the gradient estimates obtained above to derive a Harnack inequality for the heat kernel.
\begin{lemma}
    Let $G=(V,E)$ be a locally finite weighted graph and let $\psi\in C^1(0,+\infty)$ be a concave function with $H_{\psi}<\infty$. Suppose $0<g\in \mathbb{R}^{V}$. Then for any $x\sim y$, 
    \begin{equation*}
        \log\frac{g\left(y\right)}{g\left(x\right)}\leq\sqrt{H_\psi \frac{\mu_{\max}}{w_{\min}}}\sqrt{\Gamma^{\psi}\left(g\right)\left(x\right)}.
    \end{equation*}
\label{x sim y}
\end{lemma}

\begin{proof}
Lemma 5.2 in \cite{Munch} treats the special case $\mu \equiv 1$ and $w\equiv 1$.  We give a proof for the general case here. First, we claim that for any $s>0$,
\begin{equation*}
    \log s\leq \sqrt{H_\psi}\sqrt{\overline{\psi}\left(s\right)}.
\end{equation*}
Indeed, the inequality is trivial when $s\le 1$, and for $s>1$ it follows directly from the definition of $H_{\psi}$. Using this claim and the definition of $\Gamma^{\psi}$, we obtain
\begin{equation*}
\begin{aligned}
    \sqrt{\Gamma^{\psi}\left(g\right)\left(x\right)}
    &=\sqrt{\Delta\overline{\psi}\left(\frac{g} {g\left(x\right)}\right)\left(x\right)}
    =\sqrt{\frac{1}{\mu\left(x\right)}\sum_{z\sim x}w_{xz}\overline{\psi}\left(\frac{g\left(z\right)}{g\left(x\right)}\right)}\\
    &\geq\sqrt{\frac{w_{\min}}{\mu_{\max}}}\sqrt{\overline{\psi}\left(\frac{g\left(y\right)}{g\left(x\right)}\right)}
    \geq\sqrt{\frac{w_{\min}}{H_{\psi}\mu_{\max}}} \log\frac{g\left(y\right)}{g\left(x\right)}.
\end{aligned}
\end{equation*}
Rearranging gives the desired inequality.
\end{proof}

\begin{lemma}
For any $a>0,b\ge0$ and any continuous function $\gamma:[T_1,T_2]\rightarrow[0,+\infty)$, the following inequality holds
\begin{equation*}
    \min_{s\in [T_1,T_2]}\left( b\sqrt{\gamma\left(s\right)}-a\int_{s}^{T_2}\gamma\left(t\right)dt\right)\leq \frac{b^2}{a\left(T_2-T_1\right)}.
\end{equation*}
\label{t1-t2}
\end{lemma}
\begin{proof}
We refer the reader to the proof of Lemma 5.3 in \cite{Munch}.
\end{proof}

\begin{lemma}
    Let $G=(V,E)$ be a locally finite weighted graph and let $\psi\in C^{1}(0,+\infty)$ be a concave function with $\psi^\prime \left(1\right)>0$ and $H_{\psi}<\infty$.  Assume that $u:V\times[T_{1},T_{2}]\rightarrow(0,+\infty)$ satisfies
    \begin{equation*}
        \left(1-\alpha\right)\Gamma^{\psi}\left(u\right)-\psi^{\prime}\left(1\right)\frac{\partial_t u}{u}\leq\frac{c_1}{t}+c_2,
    \end{equation*}
    where $\alpha\in[0,1), c_{1}>0$ and $c_{2}\ge0$. Then for any $x,y\in V$, 
    \begin{equation*}
        u\left(x,T_1\right)\leq u\left(y,T_2\right)\left(\frac{T_2}{T_1}\right)^{\frac{c_1}{\psi^{\prime}\left(1\right)}}\exp\left(\frac{c_2}{\psi^{\prime}\left(1\right)}\left(T_2-T_1\right)+\frac{\psi^{\prime}\left(1\right)H_{\psi}\mu_{\max}d^2\left(x,y\right)}{\left(1-\alpha\right)w_{\min}\left(T_2-T_1\right)} \right).
    \end{equation*}
\label{psi harnack}
\end{lemma}

\begin{proof}
    We first consider the case where $x\sim y$. Let $s\in [T_1,T_2]$. Using the assumption on $u$, 
    \begin{equation*}
    \begin{aligned}
        &\psi^{\prime}\left(1\right)\log\frac{u\left(x,T_1\right)}{u\left(y,T_2\right)}\\
        =&\psi^{\prime}\left(1\right)\log\frac{u\left(x,T_1\right)}{u\left(x,s\right)}+\psi^{\prime}\left(1\right)\log\frac{u\left(x,s\right)}{u\left(y,s\right)}+\psi^{\prime}\left(1\right)\log\frac{u\left(y,s\right)}{u\left(y,T_2\right)}\\
        =&\int_{T_1}^{s}-\psi^{\prime}\left(1\right)\partial_t\log u\left(x,t\right)dt+\psi^{\prime}\left(1\right)\log\frac{u\left(x,s\right)}{u\left(y,s\right)}\\
        &-\psi^{\prime}\left(1\right)\int_{s}^{T_2}\partial_t\log u\left(y,t\right)dt\\
        \leq& \int_{T_1}^{s} \left(\frac{c_1}{t}+c_2-\left(1-\alpha\right)\Gamma^{\psi}\left(u\right)\left(x,t\right)\right)dt+\psi^{\prime}\left(1\right)\log\frac{u\left(x,s\right)}{u\left(y,s\right)}\\
        &+\int_{s}^{T_2}\left(\frac{c_1}{t}+c_2-\left(1-\alpha\right)\Gamma^{\psi}\left(u\right)\left(y,t\right)\right)dt\\
        \leq& \int_{T_1}^{T_2} \left(\frac{c_1}{t}+c_2\right)dt+\psi^{\prime}\left(1\right)\log\frac{u\left(x,s\right)}{u\left(y,s\right)}-\int_{s}^{T_2}\left(1-\alpha\right)\Gamma^{\psi}\left(u\right)\left(y,t\right)dt\\
        \leq& c_1\log\frac{T_2}{T_1}+c_2\left(T_2-T_1\right)\\
        &+\psi^{\prime}\left(1\right)\sqrt{H_\psi \frac{\mu_{\max}}{w_{\min}}}\sqrt{\Gamma^{\psi}\left(u\right)\left(y,s\right)}-\int_{s}^{T_2}\left(1-\alpha\right)\Gamma^{\psi}\left(u\right)\left(y,t\right)dt,
    \end{aligned}
    \end{equation*}
    where the second-to-last inequality uses  $\Gamma^{\psi}\ge0$, and the last inequality follows from Lemma \ref{x sim y}  applied to $g=u(\cdotp,s)$. Taking the minimum over $s\in [T_1,T_2]$ and applying Lemma \ref{t1-t2} with
    \begin{equation*}
        a=1-\alpha,\ \ \ b=\psi^{\prime}\left(1\right)\sqrt{H_\psi \frac{\mu_{\max}}{w_{\min}}},\ \ \ \gamma(s)=\Gamma^{\psi}\left(u\right)\left(y,s\right), 
    \end{equation*}
    we obtain  
    \begin{equation*}
        \psi^{\prime}\left(1\right)\log\frac{u\left(x,T_1\right)}{u\left(y,T_2\right)}
        \leq c_1\log\frac{T_2}{T_1}+c_2\left(T_2-T_1\right)+\frac{\left(\psi^{\prime}\left(1\right)\right)^2 H_\psi \mu_{\max}
        }{\left(1-\alpha\right)w_{\min}\left(T_2-T_1\right)}.
    \end{equation*}
    Now, let $x,y\in V$ be arbitrary. Choose a path $x=x_0\sim x_{1}\sim \cdots \sim x_k=y$ connecting $x$ and $y$. Partition the interval $[T_1,T_2]$ as $T_1=t_0<\cdots<t_k=T_2$ with $t_i-t_{i-1}=\frac{T_2-T_1}{k}$. Applying the estimate above to each adjacent pair gives
    \begin{equation*}
    \begin{aligned}
        \psi^{\prime}\left(1\right)\log\frac{u\left(x,T_1\right)}{u\left(y,T_2\right)}&=\sum_{i=1}^{k} \psi^{\prime}\left(1\right)\log\frac{u\left(x_{i-1},t_{i-1}\right)}{u\left(x_{i},t_{i}\right)}\\
        &\leq\sum_{i=1}^{k} \left[ c_1\log\frac{t_i}{t_{i-1}}+c_2\left(t_{i}-t_{i-1}\right)+\frac{\left[\psi^{\prime}\left(1\right)\right]^2 H_\psi \mu_{\max}
        }{\left(1-\alpha\right)w_{\min}\left(t_{i}-t_{i-1}\right)}\right]\\
        &=c_1\log\frac{T_2}{T_1}+c_2\left(T_2-T_1\right)+\frac{\left[\psi^{\prime}\left(1\right)\right]^2 H_{\psi}\mu_{\max}k^2}{\left(1-\alpha\right)w_{\min}\left(T_2-T_1\right)},
    \end{aligned}
    \end{equation*}
    Taking the infimum over all such paths yields the desired inequality.
\end{proof}

\begin{corollary}
    Let $G=\left(V,E\right)$ be a locally finite weighted graph satisfying the $CD\psi\left(n,-K\right)$ condition with $K\geq0$, where $\psi\in C^{1}\left(0,+\infty\right)$ is a concave function with $\psi^{\prime}(1)>0$ and $H_{\psi}<\infty$. Then for any $0<t<s$ and any $x,y,z\in V$, 
\begin{equation*}
    p\left(t,x,y\right)\leq p\left(s,x,z\right)\left(\frac{s}{t}\right)^{\frac{n}{2\psi^\prime \left(1\right)}}\exp\left(\frac{Kn(s-t)}{2\psi^{\prime}(1)}+\frac{2\psi^\prime \left(1\right)H_{\psi}\mu_{\max}d^2\left(y,z\right)}{w_{\min}\left(s-t\right)} \right).
\end{equation*}
\label{harnack heat kernel -K}
\end{corollary}

\begin{proof}
    For $K=0$, the conclusion follows from Theorem \ref{zu} by taking $W(s)=1-\frac{s}{T}$. Hence we assume $K>0$ below. By Theorem \ref{zu} with $W(t)=\frac{\sinh\left[K\left(T-t\right)\right]}{\sinh\left(KT\right)}$, we recover the Li-Xu type estimate \eqref{Li-Xu graph} from which the result follows. Set $\theta(t)=1+\coth\left(Kt\right)-Kt\ \text{csch}^{2}\left(Kt\right)$. Dividing both sides of \eqref{Li-Xu graph} by $\theta(t)$, we obtain 
    \begin{equation*}
         \frac{1}{\theta(t)}\Gamma^\psi(u)-\psi'(1)\frac{\partial_tu}{u}\leq\frac{nK}{2\theta(t)}\left[\coth\left(Kt\right)+1\right]. 
    \end{equation*}     
    A direct calculation gives
    \begin{equation*}
        \frac{K}{\theta(t)}\left[\coth\left(Kt\right)+1\right]\le\frac{1}{t}+K. 
    \end{equation*}
    Since $1\le\theta(t)<2$ and $\Gamma^{\psi}\ge0$, we have
    \begin{equation*}
        \frac{1}{2}\Gamma^\psi(u)-\psi'(1)\frac{\partial_tu}{u}\leq\frac{n}{2t}+\frac{nK}{2}.
    \end{equation*}   
    Fix $x\in V$. Applying Lemma \ref{psi harnack} to the function  $u(\tau,\cdot)=p(\tau,x,\cdot)$ with $\alpha=\frac{1}{2}$ and $c_1=\frac{n}{2}$ and $ c_2=\frac{nK}{2}$, immediately yields the desired conclusion.
\end{proof}

\section{Volume doubling}

\begin{lemma}
    Let $G=\left(V,E\right)$ be a locally finite weighted graph satisfying the $CD\psi\left(n,-K\right)$ condition with $K\geq0$, where $\psi\in C^{1}\left(0,+\infty\right)$ is a concave function. Then for any $0<f\in \ell^{\infty}\left(V\right)$ with $R_{f}<\infty$ and $\tau>0$, 
    \begin{equation*}
    \begin{aligned}
        \tau P_t\left(f\cdotp  \Gamma^{\psi}\left(f\right)\right)-\left(\tau+t\right) P_t f\cdotp\Gamma^{\psi}\left(P_t f\right)
        \geq -\psi^{\prime}\left(1\right)t\left[1+K(2\tau+t)\right]\Delta P_t f\\
        -\frac{n}{8}\left[\log\left(1+\frac{t}{\tau}\right)+4Kt+4K^{2}\tau t+2K^{2}t^{2}\right] P_t f. 
    \end{aligned}       
    \end{equation*}
\label{log ineq -K}
\end{lemma}

\begin{proof}
    We obtain the result by making a suitable choice of the parameters $\alpha$ and $\gamma$ in the variational inequality of Theorem \ref{variational th}. Set
    \begin{equation*}
        \alpha:=\tau+T-t,\ \ \ \ \ \ \ \ \ \gamma:=-\frac{n}{4\left(\tau+T-t\right)}-\frac{nK}{2}.
    \end{equation*}
    Substituting these expressions into \eqref{variational ineq} yields
    \begin{equation*}
         \partial_t\left(\alpha \phi\right)\geq-\psi^{\prime}\left(1\right)\left(1+2K\alpha\right)\Delta P_T f-\frac{n}{8}\left(\frac{1}{\alpha}+4K+4K^{2}\alpha\right) P_T f .   
    \end{equation*}
    Integrating both sides with respect to $t$ over $[0,T]$ and using
    \begin{equation*}
        \int_{0}^{T}\alpha dt=\tau T+\frac{T^2}{2}, \ \ \ \ \ \ \ \ \ \int_{0}^{T}\frac{dt}{\alpha}=\log\left(1+\frac{T}{\tau}\right),
    \end{equation*}
    we obtain    
    \begin{equation*}
    \begin{aligned}
        \alpha\left(T\right)\phi\left(T\right)-\alpha\left(0\right)\phi\left(0\right)\geq-\psi^{\prime}\left(1\right)T\left[1+K(2\tau+T)\right]\Delta P_T f\\
        -\frac{n}{8}\left[\log\left(1+\frac{T}{\tau}\right)+4KT+4K^{2}\tau T+2K^{2}T^{2}\right] P_T f. 
    \end{aligned}       
    \end{equation*}
    A direct computation gives
    \begin{equation*}
       \alpha\left(T\right)=\tau,\  \alpha\left(0\right)=\tau+T,\ \phi\left(T\right)=P_T\left(f\cdotp\Gamma^{\psi}\left(f\right)\right),\ \phi\left(0\right)=P_T f\cdotp\Gamma^{\psi}\left(P_T f\right). 
    \end{equation*}
    Therefore, we have
    \begin{equation*}
    \begin{aligned}
        \tau P_T\left(f\cdotp  \Gamma^{\psi}\left(f\right)\right)-\left(\tau+T\right)P_T f\cdotp\Gamma^{\psi}\left(P_T f\right)
        \geq -\psi^{\prime}\left(1\right)T\left[1+K(2\tau+T)\right]\Delta P_T f\\
        -\frac{n}{8}\left[\log\left(1+\frac{T}{\tau}\right)+4KT+4K^{2}\tau T+2K^{2}T^{2}\right] P_T f. 
    \end{aligned}       
    \end{equation*}
    Since $T$ is arbitrary, we may replace $T$ by $t$ to obtain the desired inequality.
\end{proof}

    For any $g\in \mathbb{R}^{V}$ with $g\le 0$ and any $\lambda>0$, we define $\varphi_{\lambda}(t,x)$ by
    \begin{equation}
        \varphi_{\lambda} \left(t,x\right):=\frac{1}{\lambda}\log P_t\left(e^{\lambda g}\right)(x),
    \label{varphi def}
    \end{equation}
\begin{lemma}
    Let $G=\left(V,E\right)$ be a locally finite weighted graph satisfying the $CD\psi\left(n,-K\right)$ condition with $K\geq0$ where $\psi\in C^{1}\left(0,+\infty\right)$ is a concave function with $\psi^\prime \left(1\right)>0$. Assume that $g\le0$ is a $d$-Lipschitz function. Then for any $\tau,\lambda>0$, there exists a constant $C\left(\lambda,d,\psi,D_{\mu}\right)>0$ depending only on $\lambda,d,\psi,D_{\mu}$ such that 
    \begin{equation*}
        \partial_t \varphi_{\lambda}\geq - \frac{1}{\psi^{\prime}\left(1\right)t}\left[\lambda C\left(\lambda,d,\psi,D_{\mu}\right)\tau+\frac{n}{8\lambda}\log\left(1+\frac{t}{\tau}\right)\right]
          -\frac{nK}{2\psi^{\prime}(1)\lambda},
    \end{equation*}
    where $\varphi_{\lambda}$ is defined by \eqref{varphi def}.  
\label{partial varphi -K}
\end{lemma}
\begin{proof}  
    From the definition of $\varphi_{\lambda}$, we have $P_t\left(e^{\lambda g}\right)=e^{\lambda\varphi_{\lambda} }$. Since $g$ is Lipschitz, the positive function $f=e^{\lambda g}$ satisfies $R_{f}<\infty$. Applying Lemma \ref{log ineq -K} to this $f$ yields
    \begin{equation*}
    \begin{aligned}
        \tau P_t\left[e^{\lambda g}\cdotp\Gamma^{\psi}\left(e^{\lambda g}\right)\right]-\left(\tau+t\right)e^{\lambda\varphi_{\lambda} }\cdotp\Gamma^{\psi}\left(e^{\lambda\varphi_{\lambda} }\right)
        \geq -\psi^{\prime}\left(1\right)t\left[1+K(2\tau+t)\right]\Delta P_t e^{\lambda g}\\
        -\frac{n}{8}\left[\log\left(1+\frac{t}{\tau}\right)+4Kt+4K^{2}\tau t+2K^{2}t^{2}\right] e^{\lambda\varphi_{\lambda} }.
    \end{aligned}       
    \end{equation*}
    By property (4) of Proposition \ref{heat} and the definition of $\varphi_{\lambda}$, we compute
    \begin{equation*}
        \Delta P_t e^{\lambda g}=\partial_t P_t e^{\lambda g}=\partial_t e^{\lambda\varphi_{\lambda} }=\lambda e^{\lambda\varphi_{\lambda} }\partial_t\varphi_{\lambda} .
    \end{equation*}
    Since $\Gamma^{\psi}\geq0$, the term $\left(\tau+t\right)e^{\lambda\varphi_{\lambda} }\cdotp\Gamma^{\psi}\left(e^{\lambda\varphi_{\lambda} }\right)$ is nonnegative. Discarding it gives the weaker inequality
    \begin{equation*}
    \begin{aligned}
         \tau P_t\left[e^{\lambda g}\cdotp\Gamma^{\psi}\left(e^{\lambda g}\right)\right]
        \geq -\psi^{\prime}\left(1\right)t\left[1+K(2\tau+t)\right]\lambda e^{\lambda\varphi_{\lambda} }\partial_t\varphi_{\lambda}\\        
        -\frac{n}{8}\left[\log\left(1+\frac{t}{\tau}\right)+4Kt+4K^{2}\tau t+2K^{2}t^{2}\right] e^{\lambda\varphi_{\lambda} }.   
    \end{aligned}
    \end{equation*}
    We now estimate the left-hand side. Because $g$ is $d$-Lipschitz, for any $x\sim y$ we have 
    $\left|g(y)-g(x)\right|\le d$. Hence, 
    \begin{equation*} 
        e^{\lambda\left(g\left(y\right)-g\left(x\right)\right)}\in\left[e^{-\lambda d},e^{\lambda d}\right]. 
    \end{equation*}
    Since the function $\overline{\psi}$ is nonnegative, attains its minimum at 1 and increases away from 1, its maximum on this interval is attained at the endpoints. Using the definition of $\Gamma^{\psi}$, 
    \begin{equation*}
    \begin{aligned}
        \Gamma^{\psi}\left(e^{\lambda g}\right)\left(x\right)=\frac{1}{\mu(x)}\sum_{y\sim x}w_{xy}\overline{\psi}\left(e^{\lambda\left[g\left(y\right)-g\left(x\right)\right]}\right)
        \leq D_{\mu}\max\left\{\overline{\psi}\left(e^{\lambda d}\right),\overline{\psi}\left(e^{-\lambda d}\right)\right\}. 
    \end{aligned}  
    \end{equation*}
    Therefore, we can choose a constant $C\left(\lambda,d,\psi,D_{\mu}\right)>0$ depending only on $\lambda,d,\psi$ and $D_{\mu}$ such that 
    \begin{equation*}
        \Gamma^{\psi}\left(e^{\lambda g}\right)\leq C\left(\lambda,d,\psi,D_{\mu}\right)\lambda^{2}.
    \end{equation*}
    Substituting this bound into the previous inequality and rearranging, we obtain
   \begin{equation*}
    \begin{aligned}
          \partial_t \varphi_{\lambda} \geq - \frac{1}{\psi^{\prime}\left(1\right)t\left[1+K(2\tau+t)\right]}\left[\lambda C\left(\lambda,d,\psi,D_{\mu}\right)\tau+\frac{n}{8\lambda}\log\left(1+\frac{t}{\tau}\right)\right]\\
          -\frac{nK\left[2+K(2\tau+t)\right]}{4\psi^{\prime}(1)\lambda\left[1+K(2\tau+t)\right]}. 
    \end{aligned}
    \end{equation*}
    To simplify the coefficients, note that $K\ge0$ implies $1+K(2\tau+t)\ge1$ and
    \begin{equation*}
        \frac{2+K(2\tau+t)}{1+K(2\tau+t)}\le2. 
    \end{equation*}
    Applying these elementary bounds, we obtain the cleaner estimate, which completes the proof.
\end{proof}
    
\begin{theorem}
     Let $G=\left(V,E\right)$ be a locally finite weighted graph satisfying the $CD\psi\left(n,-K\right)$ condition with $K\ge 0$, where $\psi\in C^{1}\left(0,+\infty\right)$ is a concave function with $\psi^\prime \left(1\right)>0, H_{\psi}<\infty $ and $\limsup _{\lambda \downarrow 0}\frac{\eta_{\psi}(\lambda)}{\lambda^{2}}<\infty $. Then there exist constants $K_0,A,c_0,C_0>0$ depending only on $n,\psi$ and $D_\mu$ such that whenever $K\leq K_0$, for $x\in V$ and $r\geq \frac{1}{2}$,
    \begin{equation}
        P_{t_r}\mathbf{1}_{B(x,r)}(x)\geq C_0 e^{-c_0\sqrt K\,r},
    \label{log heat ball -K}
    \end{equation}
    where $t_r=A\min\left\{r^2,\frac{r}{\sqrt{K}}\right\}$, with the convention that $\frac{r}{\sqrt{K}}=+\infty$ when $K=0$. 
\end{theorem}

\begin{proof}
    We obtain the result by making a suitable choice of $\tau,g$ and $t$ in Lemma \ref{partial varphi -K}. Choose $\tau$ to maximize the right hand side of the inequality in Lemma \ref{partial varphi -K}. Set
    \begin{equation*}
        \tau=\frac{t}{2}\left(\sqrt{1+\frac{n}{2\lambda^{2}C\left(\lambda,d,\psi,D_{\mu}\right)t}}-1\right),
    \end{equation*}
    With this choice, the inequality becomes
    \begin{equation*}
        -\partial_t \varphi_{\lambda}\leq\frac{1}{\psi^{\prime}\left(1\right)}\lambda C\left(\lambda,d,\psi,D_{\mu}\right)G\left(\frac{1}{\lambda^2  C\left(\lambda,d,\psi,D_{\mu}\right)t} \right)+\frac{nK}{2\psi^{\prime}(1)\lambda}, 
    \end{equation*}
    where the auxiliary function $G$ is defined by
    \begin{equation*}
        G\left(s\right):=\frac{1}{2}\left(\sqrt{1+\frac{n}{2}s}-1\right)+\frac{n}{8}s\log\left(1+\frac{2}{\sqrt{1+\frac{n}{2}s}-1}\right).
    \end{equation*}
    Integrating both sides with respect to $t$ from $t_{1}$ to $t_{2}$ gives
   \begin{equation*}
    \begin{aligned}
        \varphi_{\lambda}\left(t_1\right)\leq\varphi_{\lambda}\left(t_2\right)+\frac{1}{\psi^{\prime}\left(1\right)}\lambda C\left(\lambda,d,\psi,D_{\mu}\right)\int_{t_1}^{t_2}G\left(\frac{1}{\lambda^2  C\left(\lambda,d,\psi,D_{\mu}\right)t} \right)dt\\+\frac{nK}{2\psi^{\prime}(1)\lambda}(t_2-t_1). 
    \end{aligned}
    \end{equation*}
    Note that $G\left(s\right)\sim\sqrt{\frac{ns}{8}}$ as $s\rightarrow+\infty$ and $G\left(s\right)\rightarrow0$ as $s\rightarrow0^{+}$. Hence, the integral converges as $t_{1}\rightarrow 0$. Letting $t_1=0$ and writing $t$ in place of $t_2$, we obtain
    \begin{equation*}
    \lambda g\leq\lambda\varphi_{\lambda}\left(t\right)+\frac{1}{\psi^{\prime}\left(1\right)}\lambda^2 C\left(\lambda,d,\psi,D_{\mu}\right)\int_{0}^{t}G\left(\frac{1}{\lambda^2  C\left(\lambda,d,\psi,D_{\mu}\right)s} \right)ds+\frac{nK}{2\psi^{\prime}(1)}t.
    \end{equation*}
    Now fix $x\in V$ and choose a function $g$ with $g\left(x\right)=0$. The preceding inequality at $x$ becomes
    \begin{equation}
        1=e^{\lambda g\left(x\right)}\leq e^{\lambda\varphi_{\lambda}\left(t,x\right)}\cdotp \exp\left\{\Phi\left(\lambda^2 C\left(\lambda,d,\psi,D_{\mu}\right),t\right)\right\}\cdotp \exp\left\{\frac{nK}{2\psi^{\prime}(1)}t\right\},
    \label{phi Phi -K}
    \end{equation}
    where the function $\Phi$ is defined by
    \begin{equation*}
        \Phi\left(a,t\right)=\frac{1}{\psi^{\prime}\left(1\right)}a\int_{0}^{t}G\left(\frac{1}{as} \right)ds. 
    \end{equation*}
    
    Now take $r\ge \frac{1}{2}$ and let $B=B\left(x,r\right)$. Define  $g\left(y\right)=-d\left(y,x\right)$. Then $g\le0$ is a $1$-Lipschitz function with $g\left(x\right)=0$ and 
    \begin{equation*}
        e^{\lambda g}\leq e^{-\lambda r}\textbf{1}_{B^{C}}+\textbf{1}_{B},
    \end{equation*}
    Using the monotonicity of the heat semigroup,
     \begin{equation*}
        e^{\lambda \varphi_{\lambda}}\leq e^{-\lambda r}+P_t\left(\textbf{1}_{B}\right).
    \end{equation*}
    Substituting this into \eqref{phi Phi -K} yields
    \begin{equation*}
        P_t\left(\textbf{1}_{B}\right)\left(x\right)\geq \exp\left\{-\frac{nK}{2\psi^{\prime}(1)}t\right\} \exp\left\{-\Phi\left(\lambda^2 C\left(\lambda,1,\psi,D_{\mu}\right),t\right)\right\}-e^{-\lambda r}.
    \end{equation*}
    From the definition of $\Gamma^{\psi}$ and the Lipschitz property we can take the constant 
    \begin{equation*}
        C\left(\lambda,1,\psi,D_{\mu}\right)=\frac{1}{\lambda^2}D_{\mu}\eta_{\psi}(\lambda). 
    \end{equation*} 
    We now distinguish two cases.

    \textbf{Case 1: $Kr^2\leq1$.} Observe that by the definition of $H_{\psi}$, $\overline{\psi}\left(e^{\lambda}\right)\ge\frac{\lambda^{2}}{H_{\psi}}$, which implies $\eta_{\psi}(\lambda)\ge\frac{\lambda^{2}}{H_{\psi}}$.  
    Consequently, we obtain
    \begin{equation}
        \eta_{\psi}(\lambda)\rightarrow+\infty, \ \ \text{as}\  \lambda\rightarrow+\infty. 
    \label{lim eta}
    \end{equation}
    Since $\eta_{\psi}$ is continuous on $[0,+\infty)$ with $\eta_{\psi}(0)=0$,  by  the intermediate value theorem, we can choose $\lambda$ so that 
    \begin{equation*}
        \lambda^2C\left(\lambda,1,\psi,D_{\mu}\right)=\frac{1}{r^2}. 
    \end{equation*}
    Since $r\ge \frac{1}{2}$, we have $\eta_{\psi}(\lambda)=\frac{1}{D_{\mu}r^{2}}\le \frac{4}{D_{\mu}}$. Denote
    \begin{equation*}
        S:=\{\lambda>0:\eta_{\psi}(\lambda)\le \frac{4}{D_{\mu}}\}. 
    \end{equation*}
    By \eqref{lim eta}, the set $S$ is bounded. By assumption
    \begin{equation*}
        \limsup_{\lambda\downarrow0}
    \frac{\eta_\psi(\lambda)}{\lambda^2}<\infty,
    \end{equation*}
    there exist $\delta,M>0$ which depends only on $\psi$ such that
    \begin{equation}
        \eta_\psi(\lambda)\leq M\lambda^2,
    \qquad 0<\lambda\leq\delta.
    \label{delta}
    \end{equation}
    On the other hand, since $S$ is bounded and
    $\eta_\psi(\lambda)>0$ for $\lambda>0$, the continuity implies that $\frac{\lambda}{\sqrt{\eta_{\psi}(\lambda)}}$ has a positive lower bound on $S$ which depends only on $\psi$ and $D_{\mu}$. Therefore, we obtain
    \begin{equation*}
        \lambda r=\frac{1}{\sqrt{D_{\mu}}}\frac{\lambda}{\sqrt{\eta_{\psi}(\lambda)}}\ge c_{1}>0, 
    \end{equation*}
    where $c_{1}=c_{1}\left(\psi,D_{\mu}\right)$. It remains to estimate the term involving $\Phi$. With the change of variables $u=\frac{1}{as}$ with $a=\frac{1}{r^2}$, we obtain 
    \begin{equation*}
    \Phi\Bigl(\frac{1}{r^{2}},t\Bigr) = \frac{1}{\psi'(1)}\int_{\frac{r^{2}}{t}}^{\infty} \frac{G(u)}{u^{2}}du.
    \end{equation*}
    Set $t=A_1 r^{2}$. Then $\frac{r^{2}}{t}=\frac{1}{A_1}$, so
    \begin{equation*}
        \Phi\left(\frac{1}{r^2},A_1r^2\right)=\frac{1}{\psi^{\prime}\left(1\right)}\int_{\frac{1}{A_1}}^{\infty}\frac{G\left(u\right)}{u^2}du, \ \ \ \ \frac{nK}{2\psi^{\prime}(1)}t=\frac{nA_1}{2\psi^{\prime}(1)}Kr^2\le \frac{nA_1}{2\psi^{\prime}(1)}. 
    \end{equation*}
    Since $G\left(u\right)\sim\sqrt{\frac{nu}{8}}$ as $u\rightarrow+\infty$, 
    \begin{equation*}
        \int_{\frac{1}{A}}^{\infty}\frac{G(u)}{u^{2}}du\rightarrow0,\ \ \text{as} \ A_1\rightarrow0. 
    \end{equation*}
    Fix $\rho\in\left(0,1-e^{-c_{1}}\right)$. Therefore, we may choose $A_1$ sufficiently small, depending only on $n,\psi, D_{\mu}$ such that
    \begin{equation*}
        \exp\left\{-\frac{nK}{2\psi^{\prime}(1)}A_1r^2\right\} \exp\left\{-\Phi\left(\frac{1}{r^{2}},A_1r^2\right)\right\} \ge e^{-c_{1}} + \rho. 
    \end{equation*}
    With this choice, we have
    \begin{equation*}
        P_{A_1r^{2}}\left(\textbf{1}_{B}\right)\left(x\right)\geq \rho.
    \end{equation*}

     \textbf{Case 2: $Kr^2>1$.} Set $\lambda=L\sqrt{K}$ and $K_0=\frac{\delta^{2}}{L^2}$. Then $\lambda\leq\delta$. By \eqref{delta}, we obtain 
     \begin{equation}
         D_{\mu}\eta_{\psi}(\lambda)\leq MD_{\mu}L^{2}K. 
     \label{eta_psi}
     \end{equation}
    We claim that for $a>0$ and $s\ge 0$, 
    \begin{equation*}
        \Phi(a,s)\leq \bar{C}(n,\psi)\left[1+\log^{2}\left(1+as\right)\right]\leq\bar{C}(n,\psi)\left[1+as\right]. 
    \end{equation*}
    Set $t=\frac{A_2}{\sqrt{K}}r$. By \eqref{eta_psi}, we have
    \begin{equation*}
        \Phi\left(D_{\mu}\eta_{\psi}\left(\lambda\right),\frac{A_2}{\sqrt{K}}r\right)\leq\bar{C}\left(1+MD_{\mu}A_2\sqrt{K}rL^{2}\right), \ \ \ \frac{nK}{2\psi^{\prime}(1)}t=\frac{nA_2}{2\psi^{\prime}(1)}\sqrt{K}r. 
    \end{equation*}
    With these choices, we obtain
    \begin{equation*}
        P_{\frac{A_2}{\sqrt{K}}r}\left(\textbf{1}_{B}\right)\left(x\right)\geq e^{-\bar{C}} \exp\left\{-\left(\frac{n}{2\psi^{\prime}(1)}+\bar{C}MD_{\mu}L^2\right)A_2\sqrt{K}r\right\}-e^{-L\sqrt{K}r}. 
    \end{equation*}
    Choose first $L>0$ sufficiently large so that
    \begin{equation*}
         e^{-\frac{L}{2}}\leq \frac{1}{2}e^{-\bar{C}},
    \end{equation*}
    and then choose $A_2>0$ sufficiently small so that
    \begin{equation*}
        \left(\frac{n}{2\psi^{\prime}(1)}+\bar{C}MD_{\mu}L^2\right)A_2\leq \frac{L}{2}.
    \end{equation*}
    Since $\sqrt{K}r>1$, it follows that
    \begin{equation*}
    \begin{aligned}
         e^{-L\sqrt{K}r}&\leq e^{-\frac{L}{2}\sqrt{K}r}\exp\left\{-\left(\frac{n}{2\psi^{\prime}(1)}+\bar{C}MD_{\mu}L^2\right)A_2\sqrt{K}r\right\}\\
         &\leq\frac{1}{2}e^{-\bar{C}}\exp\left\{-\left(\frac{n}{2\psi^{\prime}(1)}+\bar{C}MD_{\mu}L^2\right)A_2\sqrt{K}r\right\}.
    \end{aligned}
    \end{equation*}
    Consequently, we have
    \begin{equation*}
        P_{\frac{A_2}{\sqrt{K}}r}\left(\textbf{1}_{B}\right)\left(x\right)\geq\frac{1}{2}e^{-\bar{C}}\exp\left\{-\left(\frac{n}{2\psi^{\prime}(1)}+\bar{C}MD_{\mu}L^2\right)A\sqrt{K}r\right\}. 
    \end{equation*}
    It remains to prove the claim. By the change of variables $u=(at)^{-1}$,
    \begin{equation*}
        \Phi(a,t)=\frac{1}{\psi'(1)}\int_{1/(at)}^\infty\frac{G(u)}{u^2}\,du.
    \end{equation*}
    From the definition of $G$, one readily obtains
    \begin{equation*}
        G(u)\leq\bar{c}\left(n\right)
    \begin{cases}
        u(1+|\log u|), & 0<u\leq1,\\
        \sqrt{u}, & u\geq1.
    \end{cases}
    \end{equation*}
    If $at\leq1$, then
    \begin{equation*}
        \Phi(a,t)\leq \frac{\bar{c}}{\psi'(1)}\int_{1/(at)}^\infty u^{-3/2}\,du\leq \frac{2\bar{c}}{\psi'(1)}.
    \end{equation*}
    If $at>1$, splitting the integral at $u=1$ gives
    \begin{equation*}
    \begin{aligned}
        \Phi(a,t)&\leq \frac{\bar{c}}{\psi'(1)}\int_{1/(at)}^1\frac{1+|\log u|}{u}\,du+\frac{\bar{c}}{\psi'(1)}\int_1^\infty u^{-3/2}\,du\\
        &\leq \frac{\bar{c}}{\psi'(1)}\left[\log(at)+\frac{1}{2}\log^2(at)+2\right].
    \end{aligned}
    \end{equation*}
    Finally, since $\log x\leq1+\log^2 x$ for $x\geq1$,
    \begin{equation*}
        \Phi(a,t)\leq \frac{3\bar{c}}{\psi'(1)}\left[\frac{1}{2}\log^2(at)+1\right].
    \end{equation*}
    In both cases, the desired estimate follows. 
\end{proof}

\begin{theorem}
    Let $G=\left(V,E\right)$ be a locally finite weighted graph satisfying the $CD\psi\left(n,-K\right)$ condition with $K\ge0$ where $\psi\in C^{1}\left(0,+\infty\right)$ is a concave function with $\psi^\prime \left(1\right)>0,\ H_{\psi}<\infty$ and $\limsup _{\lambda \downarrow 0}\frac{\eta_{\psi}(\lambda)}{\lambda^{2}}<\infty $. Then there exist  positive constants $c,C$ depending only on $n,\psi,D_\mu,\mu_{\max}, w_{\min}$ such that for any $x\in V$ and any $r>0$,
    \begin{equation*}
        V\left(x,2r\right)\leq Ce^{c\sqrt{K}r}V\left(x,r\right).
    \end{equation*}
\label{tu doubling -K}
\end{theorem}

\begin{proof}
    Fix $x\in V, K\le K_0$ and $r\ge \frac{1}{2}$. By \eqref{semigroup and kernel}, properties (6) and (7) of Proposition \ref{heat} and the Cauchy-Schwarz inequality, we obtain
    \begin{equation*}
    \begin{aligned}
         P_{t}\left({\textbf{1}}_{B\left(x,r\right)}\right)\left(x\right)&=\sum_{z\in V}\mu\left(z\right)p\left(t,x,z\right){\textbf{1}}_{B\left(x,r\right)}\left(z\right)\\
         &\leq\left(\sum_{z\in V}\mu\left(z\right)p\left(t,x,z\right)^{2}\right)^{\frac{1}{2}}\left(\sum_{z\in V}\mu\left(z\right){\textbf{1}}_{B\left(x,r\right)}\left(z\right)^{2}\right)^{\frac{1}{2}}\\
         &\leq\big{[}p\left(2t,x,x\right)\big{]}^{\frac{1}{2}}\left[V\left(x,r\right)\right]^{\frac{1}{2}}.
    \end{aligned}
    \end{equation*}
    Setting $t=t_r$, where $t_r=A\min\left\{r^2,\frac{r}{\sqrt{K}}\right\}$ is as in \eqref{log heat ball -K}, gives
    \begin{equation}
        \left[P_{t_r}\left({\textbf{1}}_{B\left(x,r\right)}\right)\left(x\right)\right]^{2}\leq p\left(2t_r,x,x\right)V\left(x,r\right).
    \label{ball r -K}
    \end{equation}
    From \eqref{log heat ball -K}, we have
    \begin{equation*}
          P_{t_r}\left({\textbf{1}}_{B\left(x,r\right)}\right)\left(x\right)\geq C_0 e^{-c_0\sqrt K r},
    \end{equation*}
    Combining this with \eqref{ball r -K} yields
    \begin{equation*}
        V\left(x,r\right)\geq\frac{C_0^2 e^{-2c_0\sqrt{K}r}}{p\left(2t_{r},x,x\right)}.
    \end{equation*}
    We claim that for any $y\in B(x,2r)$, 
    \begin{equation}
        p\left(2t_r,x,x\right)\leq H\left(n,\psi,D_{\mu},\mu_{\max},w_{\min}\right)e^{h\left(n,\psi, D_{\mu},\mu_{\max}, w_{\min}\right)\sqrt{K}r}p\left(4t_r,x,y\right). 
    \label{harnack t-2t -K}
    \end{equation}
    Upon accepting the claim, we multiply both sides by $\mu(y)$ and sum over $y\in B(x,2r)$. Using property (8) of Proposition \ref{heat}, we get
    \begin{equation*}
        p\left(2t_r,x,x\right)V\left(x,2r\right)\leq He^{h\sqrt{K}r}\sum_{y\in B\left(x,2r\right)}\mu\left(y\right)p\left(4t_r,x,y\right)\leq He^{h\sqrt{K}r}.
    \end{equation*}
    Combining this with the lower bound for $V(x,r)$ gives
    \begin{equation*}
        V\left(x,2r\right)\leq\frac{He^{h\sqrt{K}r}}{p\left(2t_r,x,x\right)}\leq \frac{H}{C_{0}^2}e^{\left(h+2c_{0}\right)\sqrt{K}r}V\left(x,r\right). 
    \end{equation*}
    It remains to prove the claim. We distinguish two cases.
    \begin{itemize}
        \item If $Kr^2\leq1$, then $t_r=Ar^2$. Applying Corollary \ref{harnack heat kernel -K} with $t=2Ar^{2}$ and $s=4Ar^{2}$, we have
    \begin{equation*}
    \begin{aligned}
         p\left(2Ar^2,x,x\right)&\leq p\left(4Ar^2,x,y\right)2^{\frac{n}{2\psi^\prime \left(1\right)}}\exp\left(\frac{An}{\psi^{\prime}(1)}Kr^2+\frac{4\psi^\prime \left(1\right)H_{\psi}\mu_{\max}}{w_{\min}A} \right)\\
         &\leq p\left(4Ar^2,x,y\right)2^{\frac{n}{2\psi^\prime \left(1\right)}}\exp\left(\frac{An}{\psi^{\prime}(1)}\sqrt{K}r+\frac{4\psi^\prime \left(1\right)H_{\psi}\mu_{\max}}{w_{\min}A} \right). 
    \end{aligned}
    \end{equation*}
    \item If $Kr^2\geq1$, then $t_r=\frac{A}{\sqrt{K}}r$. Applying the same corollary with $t=\frac{2A}{\sqrt{K}}r$ and $s=\frac{4A}{\sqrt{K}}r$, we have
    \begin{equation*}
    p\left(\frac{2Ar}{\sqrt{K}},x,x\right)\leq p\left(\frac{4Ar}{\sqrt{K}},x,y\right)2^{\frac{n}{2\psi^\prime \left(1\right)}}\exp\left(\frac{An}{\psi^{\prime}(1)}\sqrt{K}r+\frac{4\psi^\prime \left(1\right)H_{\psi}\mu_{\max}}{w_{\min}A}\sqrt{K}r \right).
    \end{equation*}
    \end{itemize} 
    In both cases, the right-hand side can be written in the form $He^{h\sqrt{K}r}p\left(4t_r,x,y\right)$ by appropriately choosing $H$ and $h$ depending only on the stated quantities. Hence,  the claim follows.
    
    We now consider the remaining two cases. First, we claim that for any $r>0$, 
    \begin{equation*}
        V(x,2r)\leq A\left(D_{\mu},\mu_{\max},w_{\min}\right)e^{a\left(D_{\mu},\mu_{\max},w_{\min}\right)r}V(x,r). 
    \end{equation*}
    Hence, if $K\ge K_{0}$, we obtain
    \begin{equation*}
        V(x,2r)\leq Ae^{ar}V(x,r)\leq Ae^{\frac{a}{\sqrt{K_0}}\sqrt{K}r}V(x,r). 
    \end{equation*}
    We now prove the claim. From the definition of $\text{deg}$, we have
    \begin{equation*}
        \left|\left\{y:y\sim x\right\}\right|\le\frac{1}{w_{\min}}\text{deg}(x). 
    \end{equation*}
    By the definition of $D_{\mu}$, we obtain
    \begin{equation*}
        \text{deg}(x)\leq D_{\mu}\mu(x)\leq D_{\mu}\mu_{\max}. 
    \end{equation*}
    Thus, we have 
    \begin{equation*}
        \left|\left\{y:y\sim x\right\}\right|\le \frac{D_{\mu}\mu_{\max}}{w_{\min}}. 
    \end{equation*}
    Using this bound, we get
    \begin{equation*}
    \begin{aligned}
        \left|B(x,2r)\right|&\le\sum_{i=1}^{[2r]}\sum_{d(x,y)=e}1\leq\sum_{i=1}^{[2r]}\left(\frac{D_{\mu}\mu_{\max}}{w_{\min}}\right)^{i}\\
        &\leq C_{1}\left(D_{\mu},\mu_{\max},w_{\min}\right)e^{C_{2}\left(D_{\mu},\mu_{\max},w_{\min}\right)R}. 
    \end{aligned}
    \end{equation*}
    Furthermore, by the definition of $\mu_{\max}$, 
    \begin{equation}
        V(x,2r)\leq\mu_{\max}\left|B(x,2r)\right|\leq C_{1}\mu_{\max}e^{C_{2}r}. 
    \label{V(2r)}
    \end{equation}
    On the other hand, from the definition of $w_{\min}$
    \begin{equation}
        V(x,r)\geq\mu(x)\geq\frac{\text{deg}(x)}{D_{\mu}}\ge\frac{w_{\min}}{D_{\mu}}.
    \label{V(r)}
    \end{equation}
    Combining \eqref{V(2r)} and \eqref{V(r)} yields
    \begin{equation*}
        V(x,2r)\leq C_1\frac{\mu_{\max}}{w_{\min}}D_{\mu}e^{C_{2}r}V(x,r). 
    \end{equation*}
    Hence, the claim follows. Finally, if $0<r\le\frac{1}{2}$, then $B(x,2r)=B(x,r)=\left\{x\right\}$. So
    \begin{equation*}
        V(x,2r)=V(x,r)=\mu(x). 
    \end{equation*}

\end{proof}

\begin{remark}
     If $\psi\in C^{2}\left(0,+\infty\right)$ is a concave function with $\psi^{\prime\prime} \left(1\right)<0$, then $\ H_{\psi}<\infty$ and $limsup _{\lambda \downarrow 0}\frac{\eta_{\psi}(\lambda)}{\lambda^{2}}<\infty$. Here we give an example illustrating that these conditions may hold even for a $C^{1}$ but not $C^{2}$ function. Define
     \begin{equation*}
        \hat{\psi}(s)=\begin{cases}
        \frac{2(1-s)^{2}}{2-s},&\ 0<s<1, \\ 
        \frac{(s-1)^{2}}{s},& \ s\ge1. 
        \end{cases} 
     \end{equation*}
    Then $\hat{\psi}$ is convex and $C^{1}$ but not $C^{2}$. Let $\psi(s)=s+2-\hat{\psi}(s)$. Then $\psi$ is concave and $C^{1}$ but not $C^{2}$.
    Obviously, $\psi^{\prime}(1)>0$ and $\bar{\psi}=\hat{\psi}$. Moreover, 
    \begin{equation*}
        H_{\psi}=\sup_{s>1}\frac{s\left(\log s\right)^{2}}{(s-1)^{2}}<\infty. 
    \end{equation*}
    Furthermore, we have
    \begin{equation*}
        \eta_{\psi}(\lambda)=\begin{cases}
        \frac{2(1-e^{-\lambda})^{2}}{2-e^{-\lambda}},&\ 0<\lambda<\log\frac{3}{2}, \\ 
        4\sinh^{2}\frac{\lambda}{2},& \ \lambda\ge\log\frac{3}{2}. 
        \end{cases} 
     \end{equation*}
    Hence,  by a direct calculus computation, we obtain
    $\limsup _{\lambda \downarrow 0}\frac{\eta_{\psi}(\lambda)}{\lambda^{2}}<\infty $. 
\end{remark}

\begin{proposition}
Let $T$ be the $d$-regular tree, $d\geq3$, equipped with unit
edge weights and counting measure, and let $\psi_0(s)=\frac{s}{1+s}$. Then, for every $n>0$, the graph $T$ satisfies
$$CD{\psi_0}\left(n,-3d+3-\frac{d}{n}\right).$$
\label{tree}
\end{proposition}

\begin{proof}
Fix $x\in T$ and let $y_1,\ldots,y_d$ be its neighbors. By homogeneity, we may assume that $f(x)=1$. Set $ a_i:=f(y_i)$ and $ A:=\sum_{i=1}^d a_i$. For each $i$, let $z_{i1},\ldots,z_{i,d-1}$ be neighbors of
$y_i$ other than $x$, and write $b_{ij}:=f(z_{ij})$. For $\psi_0(s)=\frac{s}{1+s}$, we define
\begin{equation*}
    h(a):=\psi_0(a)-\psi_0(1)
    =\frac{a-1}{2(a+1)},\qquad g(a):=\overline{\psi}_0(a)
    =\frac{(a-1)^2}{4(a+1)}.
\end{equation*}
Hence, writing $H:=\Delta^{\psi_0}f(x)$ and $\Gamma:=\Gamma^{\psi_0}(f)(x)$, we have
\begin{equation*}
    H=\sum_{i=1}^d h(a_i),
    \qquad
    \Gamma=\sum_{i=1}^d g(a_i).
\end{equation*}
A direct substitution into the definition of $\Gamma_2^{\psi_0}$ gives
\begin{equation*}
    2\Gamma_2^{\psi_0}(f)(x)=AH+\sum_{i=1}^d\left[\frac{1-a_iA}{(1+a_i)^2}+a_i h(a_i)\right]+\sum_{i=1}^d\sum_{j=1}^{d-1}F_{a_i}(b_{ij}),
\end{equation*}
where the function $F_a(b)$ is defined by
\begin{equation*}
    F_a(b):=\frac{b}{(1+a)^2}-\frac{a(b-a)}{2(a+b)}.
\end{equation*}
A direct calculation yields
\begin{equation*}
    F_a'(b)=\frac{(b-a^2)(a^2+2a+b)}{(a+1)^2(a+b)^2},
\end{equation*}
so the function $F_a$ attains its minimum at $b=a^2$. Consequently, after simplification,
\begin{equation*}
    2\Gamma_2^{\psi_0}(f)(x)\geq\frac{\left(A-d\right)^2}{4}-(A+2d-4)\Gamma+(A+d)\sum_{i=1}^d\frac{g(a_i)}{a_i+1}. 
\end{equation*}
Since $R:=\sum_{i=1}^d\frac{g(a_i)}{a_i+1}\geq0$, we obtain
\begin{equation}
    2\Gamma_2^{\psi_0}(f)(x)\geq\frac{(A-d)^2}{4}-(A+2d-4)\Gamma.
\label{eq:tree-gamma2}
\end{equation}
Moreover, since $g(a)\leq\frac{a+1}{4}$, we have $\Gamma\leq\frac{A+d}{4}$. We claim that
\begin{equation*}
    \Gamma_2^{\psi_0}(f)(x)\geq-(3d-3)\Gamma.
\end{equation*}
Indeed, by \eqref{eq:tree-gamma2}, we have
\begin{equation*}
    2\Gamma_2^{\psi_0}(f)(x)+2(3d-3)\Gamma\geq\frac{(A-d)^2}{4}+(4d-2-A)\Gamma.
\end{equation*}
If $A\leq4d-2$, the right-hand side is nonnegative. If $A>4d-2$, then, since $4d-2-A<0$, we have 
\begin{equation*}
\begin{aligned}
    &\frac{(A-d)^2}{4}+(4d-2-A)\Gamma \\
    &\qquad\geq\frac14\left[(A-d)^2+(4d-2-A)(A+d)\right]  \\
    &\qquad=\frac14\left[(d-2)A+5d^2-2d\right]>0.
\end{aligned}
\end{equation*}
This proves the claim. Finally, $h(a)^2=\frac{g(a)}{a+1}\leq g(a)$. By the Cauchy-Schwarz inequality, we obtain 
\begin{equation*}
    \bigl(\Delta^{\psi_0}f(x)\bigr)^2
    =\left(\sum_{i=1}^d h(a_i)\right)^2 \leq d\sum_{i=1}^d h(a_i)^2\leq d\,\Gamma.
\end{equation*}
It follows that
\begin{equation*}
    \Gamma_2^{\psi_0}(f)(x)\geq -(3d-3)\Gamma\geq\frac1n\bigl(\Delta^{\psi_0}f(x)\bigr)^2-\left(3d-3+\frac dn
    \right)\Gamma.
\end{equation*}
Since $x$ and $f$ were arbitrary, the desired $CD{\psi_0}$ condition follows.
\end{proof}

\begin{example}[Sharpness of the dependence on the radius]
Let $T$ be the $d$-regular tree, $d\geq3$, equipped with unit edge weights and counting measure. Let $\psi_0=\frac{s}{1+s}$. By Proposition \ref{tree}, for any $n>0$, the graph $T$ satisfies 
$$CD\psi_0(n,-3d+3-\frac{d}{n})$$ condition. Moreover, $\psi_0\in C^{2}\left(0,+\infty\right)$ is a concave function with $\psi_0^{\prime\prime} \left(1\right)=-\frac{1}{4}<0$, then $\ H_{\psi_0}<\infty$ and $limsup _{\lambda \downarrow 0}\frac{\eta_{\psi_0}(\lambda)}{\lambda^{2}}<\infty$. Thus $T$ satisfies all assumptions of Theorem \ref{tu doubling -K}. For every $x\in T$, $V(x,r)\asymp (d-1)^r$. Hence, when $r\rightarrow+\infty$, 
\begin{equation*}
    \frac{V(x,2r)}{V(x,r)}
    \asymp (d-1)^r=\exp\!\bigl(\log(d-1)r\bigr).
\end{equation*}
Consequently, the linear dependence on the radius in the exponential factor of Theorem \ref{tu doubling -K} cannot, in general, be replaced by a sublinear dependence on $r$.

\end{example}

We conclude this section with a spectral consequence of the volume
doubling property. Let
\begin{equation*}
    \ell^2(V)
    :=
    \left\{
        f\in\mathbb{R}^{V}:
        \sum_{x\in V}\mu(x)|f(x)|^2<\infty
    \right\},
\end{equation*}
equipped with the inner product
\begin{equation*}
    \langle f,g\rangle
    :=
    \sum_{x\in V}\mu(x)f(x)g(x).
\end{equation*}
Since $D_\mu<\infty$ and $w_{xy}=w_{yx}$, the operator
$-\Delta$ is a bounded nonnegative self-adjoint operator on
$\ell^2(V,\mu)$. We denote the bottom of the spectrum of $-\Delta$ by
\begin{equation*}
    \lambda_0(G)
    :=
    \inf\sigma(-\Delta)
    =
    \inf_{\substack{0\neq f\in \ell^{2}(V)}}
    \frac{\langle-\Delta f,f\rangle}
         {\langle f,f\rangle}.
\end{equation*}
For a finite set $\Omega\subset V$, let
\begin{equation*}
    w\left(\partial\Omega\right)
    :=
    \sum_{\substack{x\in\Omega,\\ y\notin\Omega}}
    w_{xy},
    \qquad
    \mu(\Omega):=\sum_{x\in\Omega}\mu(x),
\end{equation*}
and define the Cheeger constant by
\begin{equation*}
    h(G)
    :=
    \inf_{\substack{\emptyset\neq\Omega\subset V\\
    |\Omega|<\infty}}
    \frac{w\left(\partial\Omega\right)}{\mu(\Omega)}.
\end{equation*}

\begin{corollary}
    Let $G=\left(V,E\right)$ be a locally finite weighted graph satisfying the $CD\psi\left(n,0\right)$ condition where $\psi\in C^{1}\left(0,+\infty\right)$ is a concave function with $\psi^\prime \left(1\right)>0,\ H_{\psi}<\infty$ and $\limsup _{\lambda \downarrow 0}\frac{\eta_{\psi}(\lambda)}{\lambda^{2}}<\infty$. Then $h(G)=0$ and $\lambda_{0}(G)=0$. 
\end{corollary}

\begin{proof}
Fix $x\in V$. We first show that $h(G)=0$. Suppose, for contradiction, that $h(G)>0$. By the definition of $h(G)$,
\begin{equation*}
    h(G)V(x,r)\leq w\big{(}\partial B(x,r)\big{)}.
\end{equation*}
If $y\notin B(x,r)$ is adjacent to some vertex in $B(x,r)$, then $y\in B(x,r+1)\setminus B(x,r)$. Therefore, by the definition of $D_{\mu}$, 
\begin{equation*}
    w\left(\partial B(x,r)\right)
    \leq\sum_{y\in B(x,r+1)\setminus B(x,r)}\deg(y)
    \leq D_\mu \bigl(V(x,r+1)-V(x,r)\bigr).
\end{equation*}
Combining the two inequalities above, we have 
\begin{equation*}
    V(x,r+1)
    \geq
    \left(1+\frac{h(G)}{D_\mu}\right)V(x,r).
\end{equation*}
Iterating this inequality yields
\begin{equation*}
    V(x,r)\geq V(x,1)\left(1+\frac{h(G)}{D_\mu}\right)^{r-1}. 
\end{equation*}
Thus the volume grows exponentially. But by Theorem \ref{tu doubling -K} with $K=0$,
\begin{equation*}
    V(x,2r)\leq C\left(n,\psi,D_\mu,\mu_{\max}, w_{\min}\right) V(x,r).
\end{equation*}
Iterating this inequality shows that
\begin{equation*}
    V(x,r)\leq C_1 r^{\log_2 C},
\end{equation*}
for all $r\geq1$, where $C_1>0$ is independent of $r$. This contradicts the exponential lower bound above. Consequently, $h(G)=0$.

It remains to prove that $\lambda_0(G)=0$. For every nonempty finite set $\Omega\subset V$, taking $f=\mathbf{1}_\Omega$ gives $\langle-\Delta\mathbf{1}_\Omega,\mathbf{1}_\Omega\rangle=w\left(\partial\Omega\right)$ and $\langle\mathbf{1}_\Omega,\mathbf{1}_\Omega\rangle=\mu(\Omega)$. 
Thus, by the variational characterization of $\lambda_0(G)$,
\begin{equation}
    0\leq\lambda_0(G)\leq
    \frac{w\left(\partial\Omega\right)}{\mu(\Omega)}.
\end{equation}
Taking the infimum over all nonempty finite $\Omega\subset V$ yields
\begin{equation}
     0\leq\lambda_0(G)\leq h(G)=0.
\end{equation}
Therefore, $\lambda_0(G)=0$. 
\end{proof}

\begin{remark}
    A similar argument also rules out families of finite graphs satisfying the above assumptions with a uniform positive spectral gap. In particular, one obtains a $CD\psi$ analogue of the nonexistence of expanders under $CD(n,0)$ proved by M\"unch \cite{Munch 2}.
    
\end{remark}
{\bf Acknowledgements.} The author thanks Professors Bobo Hua and Bo Wu for valuable comments. 



\begin{thebibliography}{16}

\bibitem{Bakry-Emery}
Bakry, D., \'Emery, M. 
\textit{Diffusions hypercontractives. }
S\'eminaire de probabilit\'es, XIX, 1983/84, 177-206. Lecture Notes in Math., {\bf 1123} (1985), Springer, Berlin. MR0889476

\bibitem{Bakry-Qian}
Bakry, D., Qian, Z.-M.
\textit{Harnack inequalities on a manifold with positive or negative Ricci curvature.}
Rev. Mat. Iberoamericana. {\bf 15} (1999), no. 1, 143-179. MR1681640

\bibitem{Baudoin-Garofalo}
Baudoin, F., Garofalo, N.
\textit{Perelman's entropy and doubling property on Riemannian manifolds.}
J. Geom. Anal. {\bf 21} (2011), no. 4, 1119-1131. MR2836593


\bibitem{Bauer}
Bauer, F., Horn, P., Lin, Y., Lippner, G., Mangoubi, D., Yau, S.-T. 
\textit{Li-Yau inequality on graphs.}
J. Differ. Geom. {\bf 99} (2015), no. 3, 359-405. MR3316971

\bibitem{Brooks}
Brooks, R. 
\textit{A relation between growth and the spectrum of the Laplacian.}
Math. Z. {\bf 178} (1981), no. 4, 501-508. MR0638814



\bibitem{Cheng-Yau}
Cheng, S.-Y., Yau, S.-T.
\textit{Differential equations on Riemannian manifolds and their geometric applications.}
Comm. Pure Appl. Math. {\bf 28} (1975), no. 3, 333-354. MR0385749


\bibitem{Chung-Lin-Yau}
Chung, F., Lin, Y., Yau, S.-T.
\textit{Harnack inequalities for graphs with non-negative Ricci curvature.}
J. Math. Anal. Appl. {\bf 415} (2014), no. 1, 25-32. MR3172151


\bibitem{Dodziuk}
Dodziuk J. 
\textit{Difference equations, isoperimetric inequality and transience of certain random walks.}
Trans. Amer. Math. Soc. {\bf 284} (1984), no. 2, 787-794. MR0743744

\bibitem{Folz}
Folz M. 
\textit{Volume growth and spectrum for general graph Laplacians.}
Math. Z. {\bf 276} (2014), no. 1-2, 115-131. MR3150195



\bibitem{Guo}
Guo, Q., Huang, X.-P., Huang, Y.-C. 
\textit{Nonnegative Bakry-\'Emery Curvature on Bounded-Degree Graphs Implies Volume Doubling and Poincar\'e Inequalities.}
arXiv:2607.15522.


\bibitem{Grigoryan}
Grigor'yan, A. A.
\textit{Heat kernel and analysis on manifolds.}
AMS/IP Studies in Advanced Mathematics ,{\bf 47} (2009), American Mathematical Society, Providence, RI; International Press, Boston, MA. MR2569498




\bibitem{Horn}
Horn, P., Lin, Y., Liu, S., Yau, S.-T.
\textit{Volume doubling, Poincar\'e inequality and Gaussian heat kernel estimate for non-negatively curved graphs.}
J. Reine Angew. Math. {\bf 757} (2019), 89-130. MR4036571



\bibitem{Hua-Lin}
Hua, B.-B., Lin, Y. 
\textit{Stochastic completeness for graphs with curvature dimension conditions.}
Adv. Math. {\bf 306} (2017), 279-302. MR3581303


\bibitem{Jost-Liu}
Jost, J., Liu, S.-P.
\textit{Ollivier's Ricci curvature, local clustering and curvature-dimension inequalities on graphs.}
Discrete Comput. Geom. {\bf 51} (2014), no. 2, 300-322. MR3164168

\bibitem{Li-Xu}
Li, J.-F., Xu, X.-J.
\textit{Differential Harnack inequalities on Riemannian manifolds: linear heat equation.}
Adv. Math. {\bf 226} (2011), no. 5, 4456-4491. MR2270456


\bibitem{Li-Yau}
Li, P., Yau, S.-T.
\textit{On the parabolic kernel of the Schr\"odinger operator.}
Acta Math. {\bf 156} (1986), no. 3-4, 153-201. MR0834612

\bibitem{Li-Zhang}
Li, Y., Zhang, Q.-W.
\textit{Gradient estimates on graphs with the $CD\psi(n,-K)$ condition.}
J. Geom. Anal. {\bf 35} (2025), no. 10, Paper No. 324, 25 pp. MR4948694

\bibitem{Lin-Yau}
Lin, Y., Yau, S.-T.
\textit{Ricci curvature and eigenvalue estimate on locally finite graphs.}
Math. Res. Lett. {\bf 17} (2010), no. 2, 343-356. MR2644381


\bibitem{Lv-Wang}
Lv, Y., Wang, L.-F.
\textit{Gradient estimates on connected graphs with the $CD\psi(m,K)$ condition.}
Ann. Mat. Pura Appl. (4) {\bf 198} (2019), no. 6, 2207-2225. MR4031848

\bibitem{Munch}
M\"unch, F.
\textit{Li-Yau inequality on finite graphs via non-linear curvature dimension conditions.}
J. Math. Pures Appl. (9) {\bf 120} (2018), 130-164. MR3906157

\bibitem{Munch 2}
M\"unch, F.
\textit{Li-Yau inequality under $CD(0,n)$ on graphs.}
arXiv:1909.10242. 


\bibitem{Russ-Pajot}
Russ, E., Pajot, H. 
\textit{Infinite graphs satisfying the Bakry-\'Emery curvature condition $CD(0, n)$: The modified heat equation and applications to geometric analysis.}
arXiv:2507.19235.

\bibitem{Wang-Zhang}
Wang, L. F., Zhang, Z. Y.
\textit{Gradient estimates for the heat equation on graphs.}
Comm. Anal. Geom. {\bf 27} (2019), no. 4, 965-989. MR4025328



\bibitem{Zhang}
Zhang, Q. S.
\textit{A sharp Li-Yau gradient bound on compact manifolds.}
Comm. Anal. Geom. {\bf 33} (2025), no. 1, 261-274. MR4870314






\end{thebibliography}
\end{document}